\documentclass[onefignum,onetabnum]{siamart220329}
\usepackage{hyperref}
\usepackage{fullpage}
\usepackage{bm}
\usepackage{upref}
\usepackage{cleveref}
\usepackage{graphicx}
\usepackage{enumitem}
\usepackage{mathtools}
\usepackage{mathrsfs}

\newcommand{\K}{\mathbb{K}}
\newcommand{\balance}{\nabla}
\usepackage{centernot}

\newcommand{\rfield}{\mathcal{L}}
\newcommand{\vgroup}{\R}

\newcommand{\hahnseries}[2]{#1[[t^{#2}]]}

\def\R{\mathbb{R}}

\def\Rmax{\mathbb{T}}
\def\rmax{\Rmax}
\def\tmax{\Rmax}
\def\Smax{\mathbb{S}}
\def\smax{\mathbb{S}}
\def\Smaxv{\Smax^{\vee}}

\def\cK{\mathbb{K}}
\def\IsJ{I\backslash J}
\def\JsI{J\backslash I}
\def\<#1,#2>{\langle #1,#2\rangle} 
\newcommand{\semiring}{\mathcal{A}} 

\RequirePackage{amssymb}

\renewcommand{\geq}{\geqslant}
\renewcommand{\leq}{\leqslant}
\renewcommand{\preceq}{\preccurlyeq}
\renewcommand{\succeq}{\succcurlyeq}
\newcommand{\sgeq}{\succeq}
\newcommand{\sleq}{\preceq}

\newcommand{\chech}{\vee}
\newcommand{\supp}{\operatorname{supp}}
\newcommand{\sauf}[1]{\widehat{#1}}

\newcommand{\ordpolar}{\circ}
\newcommand{\sign}{\operatorname{sgn}}
\newcommand{\diag}{\operatorname{diag}}
\newcommand{\puiseuxpos}{\cK_{\geq 0}}

\newcommand{\puiseux}{\cK}

\def\0{\mathbf{0}}

\DeclareMathAlphabet{\mathbbold}{U}{bbold}{m}{n}
\newcommand{\zero}{\mathbbold{0}}
\newcommand{\unit}{\mathbbold{1}}
\def\1{\mathbf{1}}

\newcommand{\polar}{\circ}
\newcommand{\polarcouple}{\triangleright}
\newcommand{\polarsingle}{\triangleleft}
\newcommand{\signedpolar}{\circ}

\def\cp{\operatorname{CP}_n}
\def\cpk{\operatorname{CP}_{n,k}}
\def\cocp{\cp^{\polar}}

\newcommand{\psd}{\operatorname{PSD}_n}
\newcommand{\pd}{\operatorname{PD}_n}
\newcommand{\psdk}{\operatorname{PSD}_k}
\newcommand{\nn}{\operatorname{NN}_n}
\newcommand{\psdtwo}{\operatorname{PSD}_2}
\def\val{\operatorname{val}}
\def\sval{\operatorname{sval}}
\def\int{\operatorname{int}}
\def\cl{\operatorname{cl}}
\newcommand{\csdp}{\operatorname{CPSD}_n}
\newcommand{\csdpk}{\operatorname{CPSD}_{n,k}}

\newcommand{\cpsd}{\operatorname{CPSD}_n}
\newcommand{\cpsdk}{\operatorname{CPSD}_{n,k}}

\newcommand{\newplus}{\,\widehat{\oplus}\,}
\newcommand{\newplusv}{\,\widehat{\oplus}^\vee\,}

\newcommand{\nnmat}{\puiseuxpos^{n\times n}}

\newcommand{\com}{\operatorname{com}}
\newcommand{\bfA}{\mathbf{A}}
\newcommand{\bfb}{\mathbf{b}}
\newcommand{\bfx}{\mathbf{x}}

\newtheorem{example}[theorem]{Example}
\newtheorem{remark}[theorem]{Remark}

\title{Signed Tropicalization of Polar Cones}

\author{Marianne Akian, Xavier Allamigeon, St\'ephane Gaubert\thanks{Inria and CMAP, \'Ecole polytechnique, IP Paris, CNRS
    (\email{FirstName.Name@inria.fr}).}
  \and Serge\u{\i} Sergeev\thanks{University of Birmingham, School of Mathematics, Edgbaston, B15 2TT, UK
		(\email{s.sergeev@bham.ac.uk}). Corresponding author.}}

\if{
\author[M. Akian]{Marianne Akian}
\author[X. Allamigeon]{Xavier Allamigeon}
\author[S. Gaubert]{St\'ephane Gaubert}
\address{Address of Marianne Akian, Xavier Allamigeon and St\'{e}phane Gaubert: INRIA and CMAP \'{E}cole Polytechnique, IP Paris, CNRS, France}
\email{Marianne Akian, Xavier Allamigeon, St\'ephane Gaubert: FirstName.Name@inria.fr}
\author[S. Sergeev]{Serge\u{\i} Sergeev}
\address{Address of Serge\u{\i} Sergeev: University of Birmingham, School of Mathematics, Edgbaston, B15 2TT, UK}
\thanks{This work was initiated during S. Sergeev's visits to CMAP \'{E}cole Polytechnique, which were in part financially supported by 
EPSRC grant EP/P019676/1}
\email{Serge\u{\i} Sergeev: s.sergeev@bham.ac.uk}
}\fi

\begin{document}
\maketitle

\begin{abstract}
  We study the tropical analogue of the notion of polar of a cone, working over the semiring of tropical numbers with signs. We characterize the cones which arise as polars of
  sets of tropically nonnegative vectors by an invariance property
  with respect to a tropical analogue of Fourier--Motzkin elimination.
  We also relate tropical polars with
  images by the nonarchimedean valuation of classical polars over
  real closed nonarchimedean fields and show, in particular, that for semi-algebraic sets over such fields, the operation of taking the polar commutes with the operation of signed valuation
  (keeping track both of the nonarchimedean valuation and sign).
  We apply these results to characterize images
  by the signed valuation of classical cones of matrices, including the cones of
  positive semidefinite matrices, completely positive matrices, completely positive semidefinite matrices, and their polars, including the cone of co-positive matrices, showing that hierarchies of classical cones collapse under tropicalization.
  We finally discuss an application of these ideas to optimization with signed tropical numbers.
\end{abstract}



\section{Introduction}

\subsection{Motivation and context}
The present paper is motivated by a wish to develop a useful and fully functional analogue of optimization theory over the tropical semifield, which is the set of real numbers with adjoined $-\infty$, 
equipped with the operations of ``addition'' being the maximum of two numbers and ``multiplication'' being the ordinary addition. Since any real number exceeds $-\infty$, all real numbers are ``positive'' in the tropical sense, which entails that there is no tropical analogue of subtraction.

Much of tropical linear algebra, tropical convexity and their applications were developed
in a ``subtraction free'' setting, see in particular \cite{BSS07,cgq02,DS04,litvinov00,Zimm77}, and the monographs \cite{But,HOW} and \cite{JosBook}. However, the link between classical notions
(in optimization and real geometry) and their tropical analogues
can be better understood by working over an extension
of tropical numbers that encodes a sign information. This approach
was developed by M.~Plus~\cite{maxplus90b}, where the symmetrized extension of the tropical semifield
was introduced, leading to a tropical analogue of Cramer theorem: see, e.g., the monograph by Baccelli et al.~\cite{BCOQ} and the articles by Akian, Gaubert and Guterman~\cite{guterman,AGGuterman2014Cramer}. This extension involves signed tropical numbers.
The same numbers arose in Viro's patchworking method, allowing one to construct parametric families of real algebraic curves with a prescribed topology, see Itenberg and Viro~\cite{itenberg_viro}. In a nutshell, Viro's method shows that such a parametric
family admits a piecewise-linear ``log-limit'', a tropical curve.
Signed tropical numbers have been used to study the tropicalization of linear programs,
in relation with open complexity issues in game theory and linear programming, see Allamigeon et al.~\cite{ABGJ2021}
and the references therein.
Recent results in the setting of hyperfields~\cite{gunn,Jell2020,Lorsch22,Viro11}, or of semiring systems~\cite{AGRowen}, have shown a growing interest in signed tropicalizations. Recent works by Loho, Skomra and V\'egh deal with signed extensions of tropical convexity~\cite{LohoSkomra,LohoVegh}.

In classical convexity theory and in optimization, the notion of polar cone is a fundamental one,
at the heart of the duality theory. A polar cone represents the family of half-spaces containing a given set. In the tropical setting, such half-spaces are defined by linear inequalities
whose coefficients are tropical numbers with signs.
Hence, tropical polars can be thought of as sets of vectors whose components are signed tropical numbers.
This raises the question of characterizing the sets which arise
as tropical polars, as well as computing the polars of the tropical analogues
of fundamental classes of classical convex cones -- including positive semidefinite
matrices, completely positive and completely positive semidefinite matrices.
We are also interested in relating tropical polar cones with
non-archimedean valuations or ``log-limits'' of classical polar cones
defined over a real closed non-archimedean field.
We address these questions in the present paper. 

\subsection{Main results}
After recalling some of the basic definitions, we focus on the polar cones of sets of tropically positive vectors (see \Cref{s:polval}). Our first main result, \Cref{coro-polar}, characterizes these polar cones, showing they are precisely the sets of signed vectors that are closed in the topology of signed tropical numbers, and that are stable under a special operation of addition,
encoding a Fourier--Motzkin elimination rule.
We also show that polar cones are precisely the signed parts of the closed monotone precongruences initially studied by Gaubert and Katz~\cite{GK-09}. Then, we study the relation between tropical
polar cones and classical polar cones over ordered nonarchimedean fields. We
consider the notion of {\em signed valuation}, taking into account not only
the classical valuation but also the sign, already studied by Allamigeon, Gaubert and Skomra~\cite{AGS} as well as by Jell, Scheiderer and Yu~\cite{Jell2020}.
Our second main result, \Cref{t:polval}, shows that for semialgebraic subsets
over real closed non-archimedean fields with a surjective valuation
to the additive group of real numbers, the operations of taking
the polar and taking the signed valuation commute.

We next focus on the tropical analogues of classical cones (see \Cref{sec-completely}).
Here we first observe that the definitions of positive semidefinite matrices and completely positive matrices naturally extend to matrices over tropical signed numbers and we also define a tropical analogue of completely positive semidefinite matrices (see, e.g., Burgdorf, Laurent and Piovesan~\cite{Burg15} for background on completely positive semidefinite matrices). Over the field of real numbers, and more generally, over any real closed field,
the classical matrix cones form a hierarchy, with inclusions shown in \Cref{e:primal-incl} below.
Their polars, which include the cone of copositive matrices, yield a dual hierarchy displayed in \Cref{e:dual-incl}. In \Cref{t:collapse1} and \Cref{t:collapse2} we show that these hierarchies collapse under the tropicalization map, meaning that all inclusions in \Cref{e:primal-incl} and \Cref{e:dual-incl} become equalities in the tropical setting, or that the different cones in this hierarchy, over a real closed non-archimedean field, have the same images by the signed non-archimedean valuation. These results rely on surprisingly simple characterizations of tropical positive semidefinite matrices, tropical copositive matrices (see \Cref{th-psdsmax} and \Cref{t:cocptrop} below, extending a result of Yu~\cite{Yu15}). We also use a characterization of Cartwright and Chan of tropical completely positive matrices~\cite{CartChan12}.
and rely on \Cref{t:polval} (commutation between the operations of taking the signed valuation and the polar). 
In \Cref{s:optimization} we present some first elements of optimization over signed tropical numbers:
the notion of minimization turns out to be more delicate that in the classical setting,
owing to the restricted transivity of the ``order'' relation.
As 
a simple but prominent example, we give an explicit solution
of the minimization problem
for univariate polynomial in~\Cref{p:univariate}. Then, we consider in~\Cref{th-nphard} the tropical analogue
of quadratic programming. We show that testing the nonnegativity of a quadratic form on the orthant
is polynomial time solvable in the tropical world (this is perhaps surprising since the analogous problem
in the classical world is co-NP complete). However, the general tropical quadratic
optimization problem is NP-hard (\Cref{th-nphard}).

The new contributions of this paper can be thus summarized as follows: 1) tropical polar cones are introduced as sets of signed vectors and characterized as closed sets stable under a special addition operation that eliminates some coordinates; 2) tropical analogues of classical matrix cones are introduced and characterized, it is shown in particular that the classical hierarchies of matrix cones and their polars collapse under the signed nonarchimedean valuation; 3) first elements of tropical optimization over signed tropical numbers, linking in particular
the previous results with tropical quadratic optimization.

\subsection{Related work}


Tropical convexity, not taking signs into account, has been studied in great detail since the works of Zimmermann~\cite{Zimm77}, see for example \cite{AGGoub,BSS07,cgqs04,DS04,JosBook,zbMATH06293792} (among many other publications).
Tropical convex sets and cones have a lot in common with the classical convex cones and convex sets: internal representation in terms of extremal vertices, external representation as intersection of halfspaces, separation theorems, etc.

Duality is one of the topics in tropical convexity that requires a quite different approach except in some special cases (see Gavalec and Zimmermann~\cite{GavZim}), and the main reason for this is the absence of a genuine subtraction in tropical algebra. To this end, Gaubert and Katz~\cite{GK-09} introduced and characterized
{\em bi-polars}
of tropical convex sets, avoiding the use of signed tropical numbers. In this setting, bi-polars
were characterized as ``closed pre-congruences''.
\Cref{car-polar} and~\Cref{coro-polar} refine
this result, showing that a closed pre-congruence can be canonically represented
by a more concise object, the signed part of this pre-congruence, called ``signed elimination cone'',
which is characterized by a stability property with respect to Fourier--Motzkin type elimination.


Our results should be compared with recent ones by Loho and V\'egh~\cite{LohoVegh},
and then by Loho and Skomra~\cite{LohoSkomra}, studying
different types of convexity over symmetrized semirings (namely, TO-convexity, TC-convexity and hyperfield convexity). In particular, we are characterizing the tropical polars as sets that are closed and stable under a special addition operation that eliminates some coordinates, and this operation can be seen as closely related to the left-sum and TC-convexity defined in~\cite{LohoSkomra}.

The study of tropicalizations of classical matrix cones goes back to the work of Yu~\cite{Yu15},
who characterized the image by the unsigned nonarchimedean valuation of the cone
of positive semidefinite matrices. The present results refine this approach
by taking the sign information into account. We rely on results
of~\cite{AGS,Jell2020} characterizing the images by the (signed) valuation
of semialgebraic sets over ordered non-archimedean fields.
The fact that nontrivial hierarchies
of classical matrix cones collapse under the (signed) tropicalization map (\Cref{t:collapse1,t:collapse2})
is reminiscent of the fact that the Helton-Nie conjecture, which was disproved by Scheiderer~\cite{scheiderer_helton_nie},
has a valid tropical analogue as shown by Allamigeon, Gaubert and Skomra~\cite{mega2017jsc}; these different results reveal a loss
of information inherent to the tropicalization process.
We also note that the tropicalization of semialgebraic convex cones related to the moment problem was recently studied in Blekherman et al.~\cite{Yu2022}.

\section{Basic definitions}
\subsection{The symmetrized tropical semiring and its relation with real Puiseux series}

Let $\rmax$ be the {\em max-plus} or {\em tropical} semifield, that is the set of real
numbers $\R$ with $-\infty$ endowed with $\max$ operation as addition,
denoted $a\oplus b=\max(a,b)$ 
and usual addition as multiplication, denoted
$a\odot b=a+b$, with zero $-\infty$,
and unit $0$.
Recall that a {\em commutative semiring} is a set $\semiring$, 
 equipped with an addition $(a,b)\mapsto a\oplus b$ that is associative, commutative, and has a neutral element $\zero_{\semiring}$, together with a multiplication $(a,b)\mapsto a\odot b$ that is
 associative, commutative, has a unit $\unit_{\semiring}$, distributes over addition,
 and is such that $\zero_{\semiring}$ is absorbing.
It is a semifield if the nonzero elements have inverses for the multiplication.
In $\rmax$, there are no opposites for the addition, which motivated the following construction of the symmetrized semiring by M.~Plus~\cite{maxplus90b}.
The following definitions, except the order relation $\sleq$, can be found in the monograph by Baccelli et al.~\cite{BCOQ}.

Consider the set $\tmax^2:=\tmax\times \tmax$ endowed with the operations $\oplus$ and $\odot$ defined as:
\begin{align*}
&(a^+,a^-) \oplus (b^+,b^-) =(a^+\oplus b^+, a^- \oplus b^-),\\
&(a^+,a^-) \odot (b^+,b^-) = (a^+\odot b^+ \oplus a^- \odot b^-, a^+ \odot b^- \oplus a^- \odot b^+)\enspace,
\end{align*}
with $(-\infty, -\infty)$ as the zero element and $(0, -\infty)$ as the unit element, also denoted by $\zero_{\tmax^2}$ and $\unit_{\tmax^2}$, respectively.
Then, $\tmax^2$ is a semiring.
We define the operations  $\ominus$ and $|\cdot|$ 
\begin{align*}
\ominus (a^+,a^-) = (a^-, a^+),
\qquad |(a^+,a^-)| = a^+\oplus a^-\enspace.
\end{align*}
Then, $a\in \tmax^2\mapsto \ominus a\in \tmax^2$ is an additive
morphism, which allows us to 
write $a \oplus (\ominus b) = a \ominus b$ as usual.
Moreover, $a\in \tmax^2\mapsto |a|\in \tmax$ is a semiring morphism.


The semiring $\tmax$ is endowed with a natural order, which coincides with the usual order on real numbers:
\[ a\leq b\Leftrightarrow b=a\oplus b \Leftrightarrow b=a\oplus c\quad\text{for some}\;  c\enspace .\]
Based on this relation, the following relations can be defined on $\tmax^2$,
\begin{align*}
  (a^+,a^-) \sleq (b^+,b^-) &\iff a^+ \oplus b^- \leq a^- \oplus b^+ \enspace ,\\
  (a^+,a^-) \prec (b^+,b^-) &\iff a^+ \oplus b^- < a^- \oplus b^+ \enspace ,  \\
  (a^+,a^-) \balance (b^+,b^-)& \iff    a^+ \oplus b^- = a^- \oplus b^+ \enspace .
\end{align*}
For $a,b\in \tmax^2$, we say that $a$ {\em balances} $b$ when
$a \balance b$.
We define
$\tmax^2_+ := \{(a^+,a^-)\in \tmax^2 \mid a^+>a^-\} \cup\{(-\infty,-\infty)\}$,
$\tmax^2_- := \{(a^+,a^-)\in \tmax^2 \mid a^+<a^-\}\cup\{(-\infty,-\infty)\}$,
$\tmax^2_{\pm}:= \tmax^2_+\cup \tmax^2_-$.

The two relations $\sleq$ and $\balance$ are not transitive: for instance,
if $a=(5,-\infty)$, $b=(7,7)$ and $c=(4,-\infty)$, we have $a\sleq b$,
$a\balance b$,
$b\sleq c$,
$b\balance c$,
but $a\sleq c$ and $a\balance c$ do not hold.
However, we have the following properties:
\begin{proposition}
\label{p:trans}
Let $a,b,c,d\in\tmax^2$.
\begin{enumerate}[label={\rm (\roman*)}]
\item\label{p:trans-1} $a\sleq a$ for any $a\in\tmax^2$; 
\item\label{p:trans-2}  $a\sleq b$ and $b\sleq a$ if and only if $a\nabla b$;
  \item\label{p:trans-bis} $a\sleq b$ if and only if $\ominus b\sleq \ominus a$;
\item\label{p:trans-4}
$a\sleq b \text{ and }c\sleq d \implies a\oplus c\sleq b\oplus d$;
\item\label{p:trans-5}
  $a\sleq b \text{ and } c\succeq \zero_{\tmax^2}
  \implies ac \sleq bc$.
\item\label{p:trans-3} If $a\sleq b$ and $b\sleq c$ and $b\in\tmax^2_{\pm}$ then $a\sleq c$;
\end{enumerate}
\end{proposition}
\begin{proof}
  \ref{p:trans-1} to \ref{p:trans-5}: 
  Trivial.\\
\ref{p:trans-3}: As $a\sleq b$ and $b\sleq c$, we have $a^+\oplus b^-\leq a^-\oplus b^+$
and $b^+\oplus c^-\leq b^-\oplus c^+$. We consider the following three cases:

{\bf Case 1:} $b^+> b^-$ (or equivalently $b\succ \zero_{\tmax^2}$). In this case the inequalities $a\sleq b$  and  $b\sleq c$ are equivalent to
$a^+\leq a^-\oplus b^+$ and $b^+\oplus c^-\leq c^+$.
 Adding $c^-$ to the first of these inequalities, we obtain $a^+\oplus c^-\leq a^-\oplus b^+\oplus c^-$. Similarly, adding
  $a^-$ to the second inequality, we get 
  $a^-\oplus b^+\oplus c^-\leq a^-\oplus c^+$.
  Using both resulting inequalities and the transitivity of $\leq$, we deduce
  $a^+\oplus c^-\leq a^-\oplus c^+$, which is $a\sleq c$.

  {\bf Case 2:} $b^-> b^+$.
 Then, $\ominus b \succ \zero_{\tmax^2}$, and by \ref{p:trans-bis}
  $\ominus b \sleq \ominus a$  and  $\ominus  c\sleq \ominus b$.
  So using Case 1, we deduce that $\ominus  c\sleq \ominus a$,
  which is equivalent to   $a\sleq c$, by \ref{p:trans-bis}.

  {\bf Case 3:} $b^-=b^+=-\infty$: trivial.
\end{proof}
The order relation over $\tmax^2$ allows one
to denote in a compact manner inequalities
over the real numbers involving piecewise linear terms.
For instance, $a\odot b\sleq c$ can be rewritten as follows with the
usual notation:
\[
\max(a^+ + b^+, a^- + b^-, c^-) \leq \max(a^-+b^+, a^++b^-, c^+)
\enspace .
\]

Observe that \ref{p:trans-3} is a restricted transitivity property, requiring the intermediate term to belong to $\tmax^2_{\pm}$. In particular, the relations $\sleq$ and $\balance$ restricted to $\tmax^2_{\pm}$ are transitive. 
Hence, 
there is a canonical refinement of the balance relation,
that yields an equivalence relation $\mathcal{R}$ over $\tmax^2$:
\[(a^+,a^-) \mathcal{R} (b^+,b^-) \Leftrightarrow
\begin{cases}
a^+ \oplus b^- = a^- \oplus b^+& \;\text{if}\; (a^+, a^-), (b^+, b^-)\in \tmax^2_{\pm} ,\\
(a^+,a^-)=(b^+,b^-)& \text{otherwise.}
\end{cases}
\]  
One can check that $\mathcal{R}$ is compatible with the $\oplus$, $\odot$, and $\ominus$ operations, and with the relations $\sleq$, $\prec$, $\balance$, 
which are therefore defined on the quotient $\tmax^2 / \mathcal{R}$.

\begin{definition}\label{sym_def}
The \textit{symmetrized tropical semiring} is the quotient semiring $(\smax:=\tmax^2 / \mathcal{R},\oplus,\odot)$. 
We denote by $\zero:=\overline{(-\infty,-\infty)}$ the zero element 
and by $\unit:=\overline{(0, -\infty)}$ the unit element.
We set  $\smax^+=\{x\in \smax\mid x\succ \zero\}\cup\{\zero\}$,
$\smax^-=\{x\in \smax\mid x\prec \zero\}\cup\{\zero\}$,
$\smax^\circ =\{x\in \smax\mid x\balance \zero\}$,
$\smax^\vee = \smax^+ \cup \smax^-=\{\bar{x}\mid x\in \tmax^2_{\pm}\}$.
We say that the elements of $\smax^+,\smax^-$ and $\smax^\vee$ are {\em positive}, {\em negative}, and {\em signed}, respectively. 
\end{definition}
Note that here, we use the terms ``positive'' and ``negative'' in a weak sense, considering the zero element to be both positive and negative. Relating $\smax^+$ and 
$\smax^-$ with tropical pairs, it can be confirmed that $\smax^+=\tmax^2_+ / \mathcal{R}$ and $\smax^-=\tmax^2_-/\mathcal{R}$.
Beware, however, that $\{x\in \smax\mid x\succeq \zero\}=\smax^+ \cup \smax^\circ\neq \smax^+$. Observe that any element $x\in \smax^\vee$ can be written in a unique way as
$x=x^+ \ominus x^-$ with $x^+,x^-\in \smax^+$ and $x^+\odot x^-=\zero$. 
Moreover, the map $x\in \tmax\to \overline{(x,-\infty)}\in \smax^+$ 
is an isomorphism of semirings and so we identify $\tmax$ with $\smax^+$,
and $-\infty$ with $\zero$.


Note that $\preceq$ yields a total order on $\Smax^\vee$, and that $\Smax^\vee$ is unbounded
in this order. We shall complete $\Smax^\vee$ by top and bottom elements, $\top$ and $\bot$,
with the convention that $\ominus \top =\bot$, $\ominus \bot =\top$, and $\zero\odot \top=\zero\odot\bot=\zero$.


\Cref{table-1} describes when $a\nabla b$ and $a\preceq b$ in terms 
of $|a|$ and $|b|$. 

\begin{table}[ht]\small
\begin{tabular}{ccc}
\begin{tabular}{c|ccc}
$a\balance b$ & $b\in\Smax^+\setminus\{\zero\}$ & $\Smax^-\setminus\{\zero\}$ & $\Smax^{\circ}$\\ 
\hline 
$a\in\Smax^+\setminus\{\zero\}$ &  $|a|=|b|$ & never & $|a|\leq |b|$\\ 
$a\in\Smax^-\setminus\{\zero\}$ &   never & $|a|=|b|$ & $|a|\leq |b|$\\
$a\in\Smax^{\circ}$ & $|a|\geq |b|$ & $|a|\geq |b|$ & always\\
\end{tabular}
&\ \   &
\begin{tabular}{c|ccc}
$a \preceq b$ & $b\in\Smax^+\setminus\{\zero\}$ & $\Smax^-\setminus\{\zero\}$ & $\Smax^{\circ}$\\
\hline 
$a\in\Smax^+\setminus\{\zero\}$ &  $|a|\leq |b|$ & never & $|a|\leq |b|$\\
$a\in\Smax^-\setminus\{\zero\}$ &  always & $|a|\geq |b|$ & always\\
$a\in\Smax^{\circ}$ & always & $|a|\geq |b|$ & always\\
\end{tabular}
\end{tabular}\\[2mm]
\caption{Relations $a\nabla b$ (left) and $a\preceq b$ (right)}
\label{table-1}
\end{table}  
\begin{example} We have $1\succ 0\succ -1\succ \zero\succ \ominus -1\succ\ominus 0\succ \ominus 1.$
\end{example}



We now consider a field $\rfield$, equipped with a surjective nonarchimedean valuation
map $\val$, that is a map from $\rfield$ to $\R\cup\{-\infty\}$, satisfying
\begin{equation*}
\begin{aligned}
 \val(x) = -\infty &\iff x = 0 \, , \\
\forall x_{1}, x_{2} \in \rfield, \ \val(x_{1}x_{2}) &= \val(x_{1}) + \val(x_{2}) \, , \\
\forall x_{1}, x_{2} \in \rfield, \ \val(x_{1} + x_{2}) &\le \max(\val(x_{1}),\val(x_{2})) \, . \label[equation]{eq:valuation}
\end{aligned}
\end{equation*}
Note that we use the ``max-plus convention'', it is more frequent to call valuation the {\em opposite} of the map $\val$.

We also assume that $\rfield$ is equipped with a total order $\leq$,
compatible with the operations of the field,
so that $\rfield$ is an ordered field.
We require also the nonarchimedean valuation $\val$ to be \emph{convex}, meaning
that the following property is satisfied:
\begin{equation*} x_{1} , \; x_{2} \in \rfield\, ,\;\text{and}\; 
 0 \le x_{2} \le x_{1} \implies \val(x_{2})\leq  \val(x_{1})\, .
\end{equation*}
An example of ordered, and in fact real closed, field with a convex valuation
is provided by the field of real Hahn series
$\hahnseries{\R}{\vgroup}$, i.e., series with coefficients in $\R$, exponents in $\R$,
such that the opposite of the support of the series is
a well ordered set, see Engler and Prestel~\cite{engler_prestel}. Then, the valuation of a series
is its greatest exponent, and a series is positive if
the coefficient of its monomial with greatest exponent
is positive.
One may also consider the subfield $\R\{\{t^\R\}\}$ of generalized Puiseux series with real coefficients.
A non-zero element $f\in \R\{\{t^\R\}\}$ is a formal sum 
\[
f =\sum_{k\in \mathbb{N}} f_k t^{\lambda_k} \enspace,
\]
where the $f_k$ are real numbers, with $f_0\neq 0$, and $\lambda_k$ is a non-increasing sequence of real
numbers converging to $-\infty$. We have $\val(f)=\lambda_0$.
 A series $f$ is {\em positive} if
$f_0>0$, and this defines a total order on $\R\{\{t^\R\}\}$.
Markwig noted in~\cite{markwig2009field} that it is a convenient
choice of field to work out tropicalizations. 
It follows from~\cite{markwig2009field} that this field is real closed.
Another convenient choice of field is $\R\{\{t^\R\}\}_{\operatorname{cvg}}$, the
subfield of $\R\{\{t^\R\}\}$ consisting of generalized Puiseux series that
are absolutely convergent for $t>0$ large enough. This field
is also real closed, see Van den Dries and Speisseger~\cite{van_dries_power_series}. It has the advantage
that the nonarchimedean valuation also has an analytic
definition, as $\val(f)= \lim_{t\to\infty}\log (f(t))/\log(t)$.
In the sequel, for simplicity, we will make a specific choice of field,
taking $\mathbb{K}$ to be the field of formal generalized Puiseux
series $\R\{\{t^\R\}\}$. The reader can verify
that all the subsequent results carry over to any nonarchimedean real closed
field with a surjective valuation to $\R$.


Following Allamigeon, Gaubert and Skomra~\cite{AGS}, we 
 define the {\em signed valuation} of an element $x\in \rfield$
to be the element $\sval (x)\in \smax^\vee$ such that
\[ \sval (x):=\begin{cases}
\val(x) &x>0,\\
\ominus \val(x)& x<0,\\
\zero& x=0 \enspace ,
\end{cases}\]
where $\val(x)\in\tmax$ is identified as an element of $\smax^+$.
We observe that
\[
x,y\in \rfield, x\leq y \implies \sval (x) \preceq \sval(y) \enspace .
\]


As discussed in Akian, Gaubert and Rowen~\cite{AGRowen},  $\smax^\vee$   
can be identified with the 
tropical real hyperfield of Viro~\cite{viro2010hyperfields} (also called real tropical hyperfield in Jell, Scheiderer and Yu~\cite{Jell2020} or the signed tropical hyperfield in Gunn~\cite{gunn}) and $\smax$ is the hyperfield system of $\smax^{\vee}$. 
Also recall that the real tropical hyperfield may be defined
by modding out the field $\R\{\{t^\R\}\}$ by the multiplicative
group consisting of positive absolutely convergent Puiseux series
of valuation $0$. The canonical order on $\R\{\{t^\R\}\}$
defines an order relation on this hyperfield, and the relation $\preceq$
extends this order to the signed tropical numbers. However, an inconvenience of the hyperfield is that the addition is multivalued, so it is more convenient to work here with the semiring $\Smax$, understanding that all previous statements can be translated to the setting of the hyperfields.

\subsection{Classical matrix cones}
In this section, we recall the definition of several classical matrix cones.
An $n\times n$ matrix $X$ is said to be {\em nonnegative} if $X_{ij}\geq 0$
for all $i,j\in[n]$. We denote by $\nn$ the cone of nonnegative matrices.
A $n\times n$ matrix $X$ is said to be {\em completely positive}
of order $k$ if there are nonnegative vectors $y^1,\dots,y^n$ of size $k$
such that
\begin{align}
  X_{ij}=\langle y^i,  y^j\rangle , \qquad i,j\in[n] \enspace ,
  \label{e-cp}
\end{align}
where $\langle y^i,  y^j\rangle$ denotes the usual scalar product of vectors $y^i$ and $y^j$.
Equivalently,
$X=YY^T$ where $Y$ is a $n\times k$ nonnegative matrix.
We denote by $\cpk$ the set of matrices of this form,
and we denote by $\cp=\bigcup_{k\geq 1} \cpk$ the
set of {\em completely positive} matrices, which constitutes
a convex cone. The above
definitions make sense over any real closed field,
in particular over $\R$ or $\mathbb{K}$.
It will be convenient to make explicit the choice
of the field, by writing for instance $\cp(\K)$ or $\cp(\R)$.

We denote by $\psd$ the cone of positive semidefinite matrices of dimension $n\times n$. Following, 
e.g., Burgdorf, Laurent and Piovesan~\cite{Burg15}, we say that $X$ is {\em completely positive semidefinite}
of order $k$ if there exist matrices ${Y}_1,\ldots,{Y}_n\in\psdk$
such that
\begin{align} X_{ij}=\langle {Y}_i,{Y}_j\rangle,
  \qquad i,j\in [n]\enspace,
  \label{e-csdp}
\end{align}
where $\langle Y,Z\rangle:=\operatorname{tr}(YZ^T)$ denotes the Frobenius
scalar product. The set of matrices of this form is denoted by $\csdpk$, and by $\csdp$ we denote the union of these sets, which constitutes a convex cone. The representation~\eqref{e-cp} corresponds
to the special case of~\eqref{e-csdp} in which all the matrices $Y_i$ are diagonal. It follows that $\cpk\subset \csdpk$. 
The relation between the different classes of matrices considered
so far is summarized in the following table:
\begin{equation}
  \label{e:primal-incl}
  \begin{array}{ccccc}
    \csdpk&\subset& \csdp& \subset& \psd \cap \nn \\
      \cup      &       &        \cup \\
   \cpk         &    \subset   &       \cp
    \end{array} 
\end{equation}
Dually, denoting by $C^{\polar}:=\{x\mid \langle x, y\rangle\geq 0,\forall y\in C\}$
the polar of a set,
\begin{equation}
  \label{e:dual-incl}
    \begin{array}{ccccc}
      \csdpk^\polar & \supset&  \csdp^\polar& \supset & \psd + \nn \\
\cap      && \cap \\
      \cpk^\polar &\supset & \cocp
      \end{array}
\end{equation}

The inclusion $\csdpk\subset \csdp$ is known to be strict for $k=2^{O(\sqrt{n})}$ (Prakash et al.~\cite{Prakash2017}),
whereas the inclusion $\cpk\subset \cp$ is strict for $k=n^2/2 +O(n^{3/2})$ (Bomze, Schachinger and Ullrich~\cite{Bomze2015}),
see~\Cref{r:DJL} and~\Cref{rk-cpsd} for more information. It is known that
$\cp\neq \cpsd \neq\psd\cap \nn$ for $n=5$, see Fawzi et al.~\cite{Fawzi2015} for more information.
Similarly, the dual hierarchy involving the polars of these cones is strict.
Two of our main theorems imply that when taking the image by the valuation,
and for $k\geq \max(n,\lfloor{n^2/4}\rfloor)$, each of these hierarchies collapses. 

\section{Polars and valuations}
\label{s:polval}

\subsection{Tropical polars}
The polar of a set over the field of Puiseux series is defined in a standard way.

\begin{definition}
\label{d:polpui}
Let $\bm{A}$ be a subset of $\puiseux^n$. Then the {\em polar cone} of 
this set is  
\begin{equation*}
\bm{A}^{\ordpolar}=\{\bm{x}\in \puiseux^n \mid \langle \bm{x},\bm{a}\rangle \geq \bm{0} , \;\forall \bm{a}\in\bm{A}\}\enspace ,
\end{equation*}
where $\langle \bm{x},\bm{a}\rangle:=\bm{\sum}_{i=1}^n \bm{x}_i \bm{a}_i$ is the usual scalar product.
\end{definition}
In the tropical setting, we have a more delicate notion of polar introduced in \cite{GK-09}.
\begin{definition}[\cite{GK-09}]
The {\em two-sided polar} of a subset
$A\subset \Rmax^n$ is defined to be
\[
A^\polarcouple = \{(x^+,x^-) \in (\Rmax^n)^2 \mid \<x^+,a>\geq \<x^-,a>, \;\forall a\in A\},
\enspace .
\]
Dually, given a subset $B$ of $(\Rmax^n)^2$, one considers the
{\em one-sided polar} 
\[
B^\polarsingle = \{a \in \Rmax^n \mid \<x^+,a>\geq \<x^-,a>, \;\forall (x^+,x^-)\in B\}
\enspace .
\]
Here $\langle x,y\rangle=\bigoplus_{i=1}^n x_i\odot y_i$ is the usual tropical scalar product.
\end{definition}

We equip $\Rmax$ with the topology induced by the metric $d(x,y):=  |\exp(x)-\exp(y)|$. We equip $\Rmax^n$ and $(\Rmax^n)^2$ with the product topologies. We shall refer to these topologies as the {\em Euclidean} topologies.
It is known that a closed tropical convex set is the intersection
of the closed tropical half-spaces that contain it. It follows
that for all $A\subset \Rmax^n$, the ``bipolar'' $(A^\polarcouple)^\polarsingle$
is the closed convex hull of $A$, see Cohen et al.~\cite{cgqs04}. 
The characterization of the ``dual bipolar'', $(B^\polarsingle)^\polarcouple$,
requires the following notion.


\begin{definition}
\label{d:polar}
A subset of $C\in(\Rmax^n)^2$ is a {\em pre-congruence}
if the following properties hold:
\begin{enumerate}[label={\rm (\roman*)}]
\item\label{pii} If $(f^+,f^-)\in C$ then $(\lambda f^+,\lambda f^-)\in C$ for each $\lambda\in\Rmax$;
\item\label{piii} For each $(f^+,f^-)\in C$ and $(g^+,g^-)\in C$ we have $(f^+\oplus g^+, f^- \oplus g^-)\in C$;
\item\label{piv}  If $(f^+,f^-)\in C$, $(g^+,g^-)\in C$  and $f^-=g^+$, then $(f^+,g^-)\in C$. 
\end{enumerate}
We say that $C$ is a {\em monotone pre-congruence} if, in addition, the following
property holds
\begin{enumerate}[label={\rm (\roman*)}]
  \setcounter{enumi}{3}
\item\label{pi} $(f^+,f^-)\in C$ if $f^+\geq f^-$ (where $\geq$ is the standard entrywise order).
\end{enumerate}
\end{definition}
In fact, the terminology {\em polar cone} is used in Gaubert and Katz~\cite{GK-09}
for the monotone pre-congruence  notion, but since the term ``polar'' arises in different
guises here, we change for a more explicit one. We also found it convenient here to make a change of the sign convention, so that, here,
an element $(f,g)$ encodes an inequality $f\geq g$, rather than the reverse.


\begin{theorem}[{\cite[Th.~10]{GK-09}}]
  \label{t:polar:new}
  The two-sided polars of subsets of $\Rmax^n$ are precisely the closed
  monotone pre-congruences of $(\Rmax^n)^2$. In particular,
  if $C$ is a closed monotone pre-congruence, we have
  $C = (C^\polarsingle)^{\polarcouple}$.
\end{theorem}
In what follows, we will define and characterize the signed polars of subsets of $\Rmax^n$.
\begin{definition}
\label{d:polsym}
Let $A$ be a subset of $\Rmax^n$. Then the {\em signed polar} of this set is
\begin{equation}\label{e-def-signedpolar}
A^{\signedpolar}=\{x\in(\Smaxv)^n\mid \langle x,a\rangle \sgeq \zero,\;\forall a\in A\} \enspace ,
\end{equation}
where $\langle x,a\rangle = \langle x^+,a\rangle \ominus \langle x^-,a\rangle$.
\end{definition}
We recall that every element $x\in  (\Smaxv)^n$ can be written
in a unique way as $x= x^+\ominus x^-$ where $x^+,x^-\in \Rmax^n$ 
and  $x^+_i\odot x^-_i=\zero$ for all $i\in [n]$ (remembering that $\Rmax$ is identified with $\Smax^+$).
Observe that, for all $x\in (\Smaxv)^n$ and $a\in (\Rmax)^n$, 
\begin{align}
  \langle x,a\rangle \sgeq\zero \Leftrightarrow \langle x^+,a\rangle\geq \langle x^-,a\rangle \enspace .
\label{fromsigned-to-polar}
\end{align}
We next establish an equivalence between signed polars and monotone pre-congruences, see \Cref{car-polar} and \Cref{coro-polar}.
\begin{definition}
A pair $(x^+,x^-)\in(\Rmax^n)^2$ is called {\em signed} if for each $i$ we have $x_i^+ \odot x^-_i=\zero$.
\end{definition}

\begin{definition}
  Given a subset $C\subset (\Rmax^n)^2$, we denote by $C^{\chech}\subset C$ the
  {\em signed part} of $C$ defined by
\[
(x^+,x^-)\in C^{\chech}\Leftrightarrow (x^+,x^-)\in C\ \text{and}\ (x^+,x^-)\ \text{is signed.}
\]  
\end{definition}

We will also use the following notation
for the {\em diagonal} of $(\Rmax^n)^2$:
\[
\Delta^n= \{(z,z)\mid z\in \Rmax^n\}.
\]

We first show that every closed monotone pre-congruence of $(\Rmax^n)^2$ is determined
by its signed part. To this end, we introduce the following notation.
For all $(f^+,f^-)\in C$, we define $f^\chech=(f^{\chech +},f^{\chech -})$ by
the following rule:
\begin{equation*}
f^{\chech +}_i=
\begin{cases}
f^+_i, & \text{if $f^+_i\geq f^-_i$},\\
\zero, & \text{if $f^+_i<f^-_i$},
\end{cases}\quad
f^{\chech -}_i=
\begin{cases}
f^-_i, & \text{if $f^-_i> f^+_i$},\\
\zero, & \text{if $f^-_i\leq f^+_i$.}
\end{cases}
\end{equation*}
We have the following property.
\begin{lemma}\label{lem-makechech}
Let $f\in (\Rmax^n)^2$ and $x\in \Rmax^n$. Then,
\[
\langle f^+,x\rangle \geq \langle f^-,x\rangle \Leftrightarrow
\langle f^{\chech +},x\rangle \geq \langle f^{\chech -},x\rangle \enspace.
\]
\end{lemma}
\begin{proof}
  This follows from a general property of tropical half-spaces, which can be found for instance
  in~\cite[Section~3]{SGKatz}. We reproduce the argument for the reader's convenience.
  The relation $\langle f^+,x\rangle \geq \langle f^-,x\rangle $ can be rewritten
  as $\oplus_i f^+_i  x_i \geq \oplus_i f^-_i x_i $. Let $J:=\{j \mid f^+_j\geq f^-_j\}$.
  Then, the previous inequality is equivalent to
  $\oplus_{j\in J} f^+_j x_j \geq \oplus_{i\not\in J} f^-_i x_i $ which is precisely
  $\langle f^{\chech +},x\rangle \geq \langle f^{\chech -},x\rangle$.
\end{proof}
\begin{proposition}\label{cor-canchech}
  Suppose that $C\subset (\Rmax^n)^2$ is a closed monotone precongruence. Then,
  \begin{enumerate}[label={\rm (\roman*)}]
    \item\label{p-fchech} $f\in C$ implies $f^\chech\in C$;
    \item $C^\polarsingle = (C^{\chech})^\polarsingle$;
    \item  $C=C^{\chech}\oplus \Delta^n$.
      \end{enumerate}
\end{proposition}
\begin{proof}
(i).  As $C$ is a closed monotone pre-congruence, by \Cref{t:polar:new}, we have $C=A^{\polarcouple}$
  for some subset $A\subset\Rmax^n$. So $f= (f^+,f^-)\in C$ if and only if $\langle f^+,x\rangle \geq \langle f^-,x\rangle$ for all $x\in A$. Then, by~\Cref{lem-makechech}, $f^\chech \in C$.

  (ii). Since $C^{\chech} \subset C$, $C^\polarsingle \subset (C^{\chech})^\polarsingle$.
  If $a\in (C^{\chech})^\polarsingle$, then, for all $(f^+,f^-)\in C$, by (i),
  we have $(f^{\chech +},f^{\chech -})\in C^{\chech}$, and then we deduce
  from~\Cref{lem-makechech} that
  $\langle f^+,a\rangle \geq \langle f^-,a\rangle$,
  and so $a\in C^\polarsingle$.

(iii).  We argue, as in the proof of~(i),
  that $C=A^{\polarcouple}$, so that $(f^+,f^-)\in C$ if and only if $\langle f^+,x\rangle \geq \langle f^-,x\rangle$ for all $x\in A$. 

If $(f^+,f^-)\in C^{\chech}\oplus\Delta^n$ then $(f^+,f^-)=(g^+,g^-)\oplus (c,c)$ for some $(g^+,g^-)\in C^{\chech}$ and $c\in\Rmax^n$
Obviously,  $\langle f^+,x\rangle \geq \langle f^-,x\rangle$ for all $x\in A$ since $(g^+,g^-)$ and $(c,c)$ also satisfy this property. 
Thus $C^{\chech}\oplus\Delta^n\subseteq C$.

Suppose now that $f\in C=A^{\polarcouple}$. Then it follows from~\Cref{lem-makechech}
that $f^\chech\in C$.  Moreover, 
$f=f^{\chech}\oplus (c,c)$ where $c_i=\min(f^+_i,f^-_i)$ for all $i$, showing
that $C\subset C^\chech \oplus \Delta^n$.
\end{proof}

Let us introduce the following notation and operations: 

\begin{definition}
For $z\in\Rmax^n$ and $I\subseteq [n]$ we define vectors $z_I$ and $z_{\sauf{I}}$:
\[
(z_I)_i=
\begin{cases}
z_i, & \text{for $i\in I$},\\
\zero 
\end{cases},\quad
(z_{\sauf{I}})_i=
\begin{cases}
z_i, & \text{for $i\notin I$},\\
\zero 
\end{cases}
\]
For $I=\{i\}$ we denote $z_{\sauf{i}}=z_{\sauf{\{i\}}}$.
\end{definition}

\begin{definition}
  Let $(x^+,x^-)\in(\Rmax^n)^2$ and $(y^+,y^-)\in(\Rmax^n)^2$ and suppose that
  $x^-_i=y^+_i$ for some $i\in [n]$. Then we define
\[
(x^+,x^-)\oplus_i (y^+,y^-)=(x^+\oplus y^+_{\sauf{i}},\; x^-_{\sauf{i}}\oplus y^-). 
\]
\end{definition}
\begin{definition}
  Let $(x^+,x^-)\in(\Rmax^n)^2$ and $(y^+,y^-)\in(\Rmax^n)^2$, and let $I=\{i\in[n]\mid x^-_i = y^+_i\}$.
Then we define
\begin{align*}
&(x^+,x^-)\newplus (y^+,y^-)=(x^+\oplus y^+_{\sauf{I}},\; x^-_{\sauf{I}}\oplus y^-) \enspace ,\\
\text{and}\quad  & (x^+,x^-)\newplusv (y^+,y^-)= \left((x^+,x^-)\newplus (y^+,y^-)\right)^\chech \enspace . 
\end{align*}
\end{definition}
In particular, if $I$ is empty, $(x^+,x^-)\newplus (y^+,y^-)
=(x^+,x^-){\oplus} (y^+,y^-)$.
\begin{remark}
If we identify $(\Rmax^n)^2$ to $(\Rmax^2)^n$ by  $(x^+,x^-)\mapsto ((x_i^+,x_i^-))_{i=1,\ldots , n}$,
we get that the operations $\newplus$ and $\newplusv$ can be defined componentwise: $x\newplus y= ((x_i\newplus y_i))_{i=1,\ldots, n}$ for all
$x=(x_i)_{i=1,\ldots, n}$, $y=(y_i)_{i=1,\ldots, n}$, with $x_i,y_i\in \Rmax^2$, for all $i=1,\ldots, n$. 
Also the operation $\newplusv$ can be defined equivalently as  
a binary operation on $\Smax^\vee$.
\end{remark}


\begin{definition}
\label{t:spolar}
A set $R\subset(\Rmax^n)^2$ is called a {\em signed elimination cone}
if each pair $(x,y)\in R$ is signed and the following properties hold:
\begin{enumerate}[label={\rm (\roman*)}]
\item \label{bendi}$(x^+,\zero)\in R$ for each $x^+\in\Rmax^n$;
\item \label{bendii}If $(x^+,x^-)\in R$ then $(\lambda x^+,\lambda x^-)\in R$ for each $\lambda\in\Rmax$;
\item \label{bendiii} For each $(x^+,x^-)\in R$ and $(y^+,y^-)\in R$ we have $((x^+,x^-)\oplus (y^+,y^-))^{\chech}\in R$;
\item\label{bendiv} For each $(x^+,x^-)\in R$ and $(y^+,y^-)\in R$  and $i$ such that $x^-_i=y^+_i$, 
we have $((x^+,x^-)\oplus_i (y^+,y^-))^{\chech}\in R$.
\end{enumerate}
\end{definition}
We will see below (\Cref{p:bendequiv}) that the last two properties can be replaced by the stability of $R$ under the $\newplusv$
operation. \Cref{bendiv} may be interpreted in terms of ``Fourier--Motzkin elimination'', as
it can be seen in \Cref{ex:FM1} and \Cref{ex:FM2} below.
The terminology ``signed elimination cone'' refers to such interpretation.
The special addition appearing in \Cref{bendiv} is also reminiscent of the bend relation (Giansiracusa and Giansiracusa~\cite{Giansiracusa2016}, Maclagan and Rinc\'{o}n~\cite{Maclagan2018}), which plays a somehow similar role
in the unsigned setting.

We shall identify $(\Rmax^n)^2$ to $(\Rmax^2)^n$. In this way
a subset of $(\Rmax^n)^2$ consisting only of signed pairs
can be identified to a subset of $(\tmax^2_{\pm})^n$
or equivalently to a subset of $(\Smaxv)^n$, and vice
versa. We shall denote by $\imath$ the canonical
identification map from $(\Smaxv)^n$ to the set of signed
pairs of $(\Rmax^n)^2$. In particular,
the {\em signed polar} $U = A^\circ$ of a subset $A\subset \Rmax^n$
which was defined as a subset of $(\Smaxv)^n$ gives
rises to subset $\imath(U) \subset (\Rmax^n)^2$
consisting of signed elements. We also recall that $(\Smaxv)^n$ is equipped with the product topology
defined by considering the order topology on $\Smaxv$.

\begin{theorem}\label{car-polar}
  Let $R\subset (\Rmax^n)^2$. The following assertions are equivalent:
  \begin{enumerate}[label={\rm (\roman*)}]
  \item $R$ is a closed signed elimination cone.
  \item There exists a closed monotone pre-congruence $C$ of $\Rmax^n$
    such that $R$ is the signed part of $C$.
   \end{enumerate}
  \end{theorem}
\begin{theorem}\label{coro-polar}
  Let $U\subset (\Smaxv)^n$. The following assertions are equivalent:
  \begin{enumerate}[label={\rm (\roman*)}]
  \item The image $\imath(U)\subset (\Rmax^n)^2$ by the canonical
    map is a signed elimination cone, and $U$ is closed in $(\Smaxv)^n$.
  \item There exists a subset $A\subset \Rmax^n$ such that $U= A^\circ$.
  \end{enumerate}
  \end{theorem}

The proof of~\Cref{car-polar} and of~\Cref{coro-polar}
relies on a series of auxiliary results.

\begin{proposition}\label{prop-precongtobend}
Let $C$ be a closed monotone pre-congruence. Then $C^{\chech}$ is a signed elimination cone.
\end{proposition}
\begin{proof}
We need to check that $C^{\chech}$ satisfies the properties \ref{bendi}-\ref{bendiv} of \Cref{t:spolar}.

\ref{bendi}: if $C$ is a monotone pre-congruence, then, by ~\Cref{d:polar}, (iv), for all $x^+\in \Rmax^n$, $(x^+,\zero)\in C$, and since $(x^+,\zero)$ is signed, $(x^+,\zero)\in C^{\chech}$.

\ref{bendii}: obvious from \Cref{d:polar} part~\ref{pii}.

\ref{bendiii}: suppose that $(x^+,x^-)\in C^\chech$ and $(y^+,y^-)\in C^\chech$.
Then, by~\Cref{d:polar}, \ref{piii}, we have $(x^+,x^-)\oplus (y^+,y^-)\in C$.
Then, by~\Cref{cor-canchech} part \ref{p-fchech}, $((x^+,x^-)\oplus (y^+,y^-))^\chech \in C^\chech$.


\ref{bendiv}: Let $(x^+,x^-)\in C^{\chech}$ and $(y^+,y^-)\in C^{\chech}$  with $x^-_i=y^+_i=t$. Then we have
\[
x^-=x^-_{\sauf{i}}\oplus te_i,\quad y^+=y^+_{\sauf{i}}\oplus te_i.
\] 
Our purpose is to show that $(x^+\oplus y^+_{\sauf{i}},x^-_{\sauf{i}}\oplus y^-)^{\chech}\in C^{\chech}$.
We have $(x^+,x^-)\in C$ and by monotonicity of $C$, $(x^-,x^-_{\sauf{i}})\in C$, hence,
by transitivity of $C$, $(x^+,x^-_{\sauf{i}})\in C$. Similarly,
$(x^-,te_i)\in C$ by monotonicity, and so, $(x^+,t e_i)$ holds by transitivity.

As $(y^+_{\sauf{i}},y^+_{\sauf{i}})\in C$, we can add it to
$(x^+,te_i)$ and then we obtain $(x^+\oplus y^+_{\sauf{i}}, y^+)\in C$.
Using transitivity again, we obtain
$(x^+\oplus y^+_{\sauf{i}}, y^-)\in C$. 

As $(x^+,x^-_{\sauf{i}})\in C$, we can add it to $(x^+\oplus y^+_{\sauf{i}},y^-)$ and obtain 
$(x^+\oplus y^+_{\sauf{i}}, y^-\oplus x^-_{\sauf{i}}) \in C$. Using again~\Cref{cor-canchech} part \ref{p-fchech},
we deduce that $(x^+\oplus y^+_{\sauf{i}}, y^-\oplus x^-_{\sauf{i}})^\chech \in C^\chech$. 
\end{proof}

For all $z\in \Rmax^n$, we denote by $\supp(z)=\{i\in [n]\mid z_i\neq \zero\}$
the {\em support} of $z$. In what follows we will also consistently drop the $\odot$ sign.

\begin{lemma}
\label{l:decreasey}
Let $R$ be a signed elimination cone and take any $x=(x^+,x^-)\in R$ and suppose
that $y=(y^+,y^-)$ is signed and that $y\succeq x$.
Then $y\in R$.
\end{lemma}
We note that,
\begin{align}
\text{for }x \text{ and }y\text{ signed, }
y\succeq x \iff (y^+\geq x^+ \text{ and } x^- \geq  y^-) \enspace .\label{e-reformulate}
\end{align}
\begin{proof}
  We first show that $(x^+,x^-_ie_i)\in R$ for all $i\in \supp(x^-)$.
  More generally, we will show that
  $(x^+,x^-_{\sauf{I}})\in R$ for all subsets $I\subset \supp(x^-)$.
 For this, it is enough to show 
that if $(x^+,x^-_{\sauf{I}})\in R$ then  $(x^+,x^-_{\sauf{I\cup\{j\}}})\in R$ for any $j\notin I$. 
But this follows since $(x^-_je_j,\zero)\in R$ and since
\[
(x^+,x^-_{\sauf{I\cup\{j\}}})=(x^+,x^-_{\sauf{I}})\oplus_j (x^-_je_j,\zero) \in R 
\]
In this way, by induction (successively increasing $I$ to $[n]\backslash\{i\}$), we conclude that
that $(x^+,x^-_ie_i)\in R$ for all $i\in\supp(x^-)$.

We also observe that $(x^+,y^-_i e_i)\in R$, 
as for the case when $y^-_i\neq -\inf$ we have that $x_i^-\geq y_i^-$ and $(x^+,y^-_i e_i)=(x^-_i)^{-1}  y^-_i  (x^+,x^-_ie_i)\oplus (x^+,\zero)$. 

It follows that $(x^+,y^-_ie_i)\in R$ for all $i\in \supp(y^-)$, and then $(x^+,y^-)\in R$ as 
$(x^+,y^-)=\bigoplus_{i=1}^n (x^+,y^-_ie_i)$.

Finally, $(y^+,y^-)=(x^+,y^-) \oplus (y^+,\zero)\in R$.  
\end{proof}

\begin{lemma}
\label{l:cancel}
Let $R$ be a signed elimination cone. Then, $x,y\in R$ implies $(x\newplus y)^{\chech}\in R$.
\end{lemma}
\begin{proof}
  We introduce vectors $u,v,z$ such that $u\oplus z=x^-$, $v\oplus z = y^+$
  and supports of $v$ and $z$, as well as the supports of $u$ and $z$, are pairwise disjoint. In this new notation
 $(x\newplus y)^{\chech}=(x^+\oplus v, u\oplus y^-)^{\chech}$. 

We will prove by induction that $(x^+\oplus v\oplus z_{\sauf{I}},u\oplus y^-)^{\chech}\in R$ for all $I\subseteq [n]$. 
The basis of induction is $(x^+\oplus v\oplus z,u\oplus y^-)^{\chech}\in R$, which is equal to  
$((x^+,u)\oplus (v\oplus z,y^-))^{\chech}$. Here $(x^+,u)\in R$ by \Cref{l:decreasey}. 

We need to show that  if $(x^+\oplus v\oplus z_{\sauf{I}},u\oplus y^-)^{\chech}\in R$ then 
$(x^+\oplus v\oplus z_{\sauf{I\cup\{j\}}},u\oplus y^-)^{\chech}\in R$  for any $j\notin I$. 
For this, first note that $(x^+,u\oplus z_{\sauf{I}})\in R$ by Lemma~\ref{l:decreasey}, and that 
all components of $z_{\sauf{I}}$ are present in $(x^+\oplus v\oplus z_{\sauf{I}},u\oplus y^-)^{\chech}$ since 
the supports of $z$ and $u$ are disjoint by the definition of these vectors and the supports of $z$ and $y^-$ are disjoint 
since $(v\oplus z,y^-)\in R$. 
Then we observe that
\begin{equation*}
\begin{split}
& \left((x^+,u\oplus z_{\sauf{I}})\oplus_j (x^+\oplus v\oplus z_{\sauf{I}},u\oplus y^-)^{\chech}\right)^{\chech}\\
&= (x^+\oplus v\oplus z_{\sauf{I\cup\{j\}}},u\oplus y^-\oplus  z_{\sauf{I\cup\{j\}}})^{\chech}=\\
&= (x^+\oplus v\oplus z_{\sauf{I\cup\{j\}}},u\oplus y^-)^{\chech}.
\end{split}
\end{equation*}
Therefore  $(x^+\oplus v\oplus z_{\sauf{I\cup\{j\}}},u\oplus y^-)^{\chech}\in R$, as claimed.

Setting $I=[n]$ we obtain $(x^+\oplus v, u\oplus y^-)^{\chech}\in R$. As  $(x^+\oplus v, u\oplus y^-)^{\chech}=(x\newplus y)^{\chech}$, the proof is complete.
\end{proof}

\begin{proposition}
\label{p:bendequiv}
  Suppose that $R\subset (\Rmax^n)^2$ consists of signed vectors and satisfies 
\ref{bendi} and \ref{bendii}
  of~\Cref{t:spolar}. Then the conditions \ref{bendiii} and \ref{bendiv} of this definition hold if and only if
  $R$ is stable under the $\newplusv$ operation.
\end{proposition}
\begin{proof}
  The ``only if'' part follows from~\Cref{l:cancel}.

  Conversely,  using \Cref{l:decreasey}, it suffices
  to note that $(x\newplus y)^\chech \preceq (x \oplus_i y)^\chech
  \preceq  (x \oplus y)^\chech$, which follows using~\eqref{e-reformulate} by comparing the positive parts
  (the negative parts remain the same).
  \end{proof}
\begin{remark}
The above result means that $R\subset (\Rmax^n)^2$ is a signed elimination cone if and only if
it contains $\tmax^n$ and is a convex cone with respect to the $\odot$ scalar 
multiplication and $\newplusv$ addition.
\end{remark}

\begin{proposition}
\label{p:transitivity}
Let $R\in(\Rmax^n)^2$ be a signed elimination cone and define $C=R\oplus \Delta^n$. Then 
$C$ is a monotone pre-congruence.
\end{proposition}
\begin{proof}
We have to prove that $C$ satisfies all the properties listed in \Cref{d:polar}.

\Cref{d:polar}, \ref{pi}: For $(f^+,f^-)$ such that $f^+\geq f^-$ we can write $(f^+,f^-)=(f^+,\zero)\oplus (f^-,f^-)$.
By \Cref{t:spolar} \ref{bendi} we have $(f^-,\zero)\in R$, hence $(f^+,f^-)\in C$.

\Cref{d:polar}, \ref{pii}: Follows from \Cref{t:spolar} \ref{bendii}.

\Cref{d:polar} \ref{piii}: We have 
\begin{equation*}
(f^+,f^-)\oplus (g^+,g^-)=((f^+,f^-)^{\chech}\oplus (g^+,g^-)^{\chech})^{\chech}\oplus (d,d)\enspace .
\end{equation*} 
where $d_i=\min (f^+_i\oplus g^+_i,\; f^-_i\oplus g^-_i)$ for all $i$.
It suffices to check this property for the case when $f^+,f^-,g^+$ and $g^-$ are scalars, which is elementary.
It follows then that $(f^+,f^-)\oplus (g^+,g^-)\in C$ since $(f^+,f^-)^{\chech}\in R$ and $(g^+,g^-)^{\chech}\in R$ (note that if we represent 
$(f^+,f^-)=(x^+,x^-)\oplus (c,c)$ with $(x^+,x^-)\in R$ then $(f^+,f^-)^{\chech}\succeq (x^+,x^-)$). 

\Cref{d:polar}, \ref{piv}: Let us represent 
\[
(f^+,f^-)=(x^+,x^-)\oplus (c,c),\quad (g^+,g^-)=(y^+,y^-)\oplus (d,d),
\]
where $(x^+,x^-)\in R$, $(y^+,y^-)\in R$ and $f^-=g^+$.

Denote by $I$ the set of indices for which $x^-_i=f^-_i$. 
Then we can write 
\[
(f^+,f^-)=(x^+,f^-_I)\oplus (c,c),
\]
where $(x^+,f^-_I)\in R$ by \Cref{l:decreasey} (since $(x^+,x^-)\in R$ and $x^-\geq f^-_I$). We can also write 
\[
(g^+,g^-)=(y^+\oplus d,y^-)^{\chech}\oplus (d,d)=(g^+_J,z)\oplus (d,d)
\]
for some index set $J\subseteq [n]$ 
and some $z$ such that  $(y^+\oplus d,y^-)^{\chech}=(g^+_J,z)$ (it is easy to check that the nonzero components of the positive part of $(y^+\oplus d,y^-)^{\chech}$ are components of $g^+$).  Here we also have that $(g^+_J,z)=(y^+\oplus d,y^-)^{\chech}=((y^+,y^-)\oplus (d,\zero))^{\chech}$
is in $R$ by \Cref{t:spolar} \ref{bendiii}. We then can write
\begin{equation}
\label{e:fgrep}
(f^+,f^-)=(x^+,f^-_I)\oplus (c,c),\quad (g^+,g^-)=(g^+_J,z)\oplus (d,d).
\end{equation}
and denote $K=I\cap J$ and $L=\sauf{I\cup J}$ (the complement of $I\cup J$).

By \Cref{l:cancel} we obtain $((x^+,f^-_I)\newplusv (g^+_J,z))\in R$, hence\\
$((x^+,f^-_I)\newplus (g^+_J,z)) = (x^+\oplus  g^+_{\JsI}, f^-_{\IsJ}\oplus z)\in C$. 
From~\eqref{e:fgrep} and $f^-=g^+$, we have $f^-_I\oplus c=g^+_J\oplus d$, and therefore
\begin{equation}
\label{e:g1}
c_{\IsJ}\oplus f^-_{\IsJ}=d_{\IsJ},\quad c_{\JsI}=d_{\JsI}\oplus g^+_{\JsI},
\end{equation}
and
\begin{equation}
\label{e:g2}
\quad c_K\oplus f^-_K=d_K\oplus g^+_K,\quad c_L=d_L \enspace .
\end{equation}
Then \Cref{e:g1} together with $C\oplus \Delta^n\subset C$ implies that
\begin{equation*}
\begin{split}
(x^+\oplus  g^+_{\JsI}, &f^-_{\IsJ}\oplus z)\oplus (c_{\IsJ}\oplus d_{\JsI} ,c_{\IsJ} \oplus d_{\JsI})\\
&=(x^+\oplus c_{\IsJ}\oplus c_{\JsI}, z\oplus d_{\IsJ}\oplus d_{\JsI})\in C.
\end{split}
\end{equation*}
\Cref{e:g2} implies that for any $i\in K$ we have either $c_i=d_i>f^-_i(=g^+_i)$ or 
$c_i\leq f^-_i(=g^+_i)$ and $d_i\leq f^-_i(=g^+_i)$, and that we have $c_i=d_i$ for any $i\in L$.
Define the sets 
\begin{equation*}
K_1= \{i\in K \colon c_i=d_i>f^-_i\}\cup L,\quad K_2=\{i\in K\colon c_i\leq f^-_i\; \text{and}\; d_i\leq f^-_i\} \enspace .
\end{equation*}
Note that $\IsJ\cup\JsI\cup K_1\cup K_2=[n]$, and that we have $c_{K_1}=d_{K_1}$. 
Also as $(x^+, f^-_I)\in R$ and $f^-_I\geq f^-_{K_2}\geq d_{K_2}$, we obtain that $(x^+,d_{K_2})\in R$ by \Cref{l:decreasey}. 
 
We then show, using \Cref{e:fgrep} and the facts established above, that $(f^+,g^-)\in C$ since
\begin{equation*}
\begin{split}
&(f^+, g^-) = (x^+\oplus c, z\oplus d)\\
&= (x^+\oplus c_{\IsJ}\oplus c_{\JsI}\oplus c_{K_1}\oplus c_{K_2}, z\oplus d_{\IsJ}\oplus d_{\JsI}\oplus d_{K_1}\oplus d_{K_2})\\
& =(x\oplus c_{\IsJ}\oplus c_{\JsI}, z\oplus d_{\IsJ}\oplus d_{\JsI})\oplus (c_{K_1},c_{K_1})\oplus (c_{K_2},0)\oplus  (x^+,d_{K_2}).
\end{split}
\end{equation*}
This completes the proof. 
\end{proof}

\begin{lemma}\label{lem-topcoincide}
  Let $U \subset (\Smaxv)^n$ and $R:=\imath(U)\subset (\Rmax^n)^2$.
  Then, $U$ is closed in the product of the order topology
  on $\Smaxv$ iff $R$ is closed in the topology induced
  by the Euclidean topology on $(\Rmax^n)^2$.
\end{lemma}
\begin{proof}
The Euclidean topology
on $(\Rmax^n)^2$ is defined as a product topology,
and $(\Smaxv)^n$ is also equipped with a product topology,
so it suffices to show the result when $n=1$. Then
it is elementary to verify that the order topology on $\Smaxv$
coincides with the topology induced by the Euclidean
topology on the set of signed pairs of $\Rmax^2$ -- indeed,
these are two realizations
of the same topological space, obtained by gluing two copies of the half-line
$[-\infty,\infty)$ (equipped with the order topology) at point $-\infty$. 
 \end{proof}
\begin{proof}[Proof of~\Cref{car-polar}]
  (2)$\Rightarrow$(1).
  If $C$ is a closed monotone pre-congruence, then the signed
  part of $C$, $R=C^{\chech}$, is a signed elimination cone,
  by~\Cref{prop-precongtobend}. Moreover,
  observe that the subset of signed pairs
  of $(\rmax^n)^2$ is closed. Hence,
  $C^{\chech}$, which is the intersection of the closed set $C$ with this subset,
  is also closed. 

    (1)$\Rightarrow$(2). If $R$ is a signed elimination cone, then,
  by~\Cref{p:transitivity}, $C= R\oplus \Delta^n$ is a monotone
  pre-congruence. We have also that $R$ is the signed part of $C$.
  Moreover, suppose that $R$ is closed.
  Let $c_k = r_k \oplus \delta_k$ with $r_k\in R$ and $\delta_k\in \Delta^n$
  be a sequence of elements of $C$ converging to $c\in (\Rmax^n)^2$. Then the
  sequences $r_k$ and $\delta_k$ are bounded from above. By taking
  subsequences, we may assume that $r_k$ converges to some $r$
  and that $\delta_k$ converges to some $\delta$. Since $\Delta^n$
  and $R$ are closed, we deduce that $c=r \oplus \delta\in C$.
  So, $C$ is closed.
  \end{proof}
\begin{proof}[Proof of~\Cref{coro-polar}]
    (2)$\Rightarrow$(1). If $U=A^\circ$ for some subset $A\subset \Rmax^n$, then,
  by~\eqref{fromsigned-to-polar}, $R:=\imath(U)$ is the signed part of the polar $A^\polarcouple$.
  By~\Cref{t:polar:new}, $A^\polarcouple$ is a closed monotone pre-congruence.
  Hence, by~\Cref{car-polar}, (2) $\Rightarrow$ (1), $R$ is a signed elimination cone
  that is closed in the topology of $(\Rmax^n)^2$, i.e., the Euclidean topology.
  Then, by~\Cref{lem-topcoincide}, $\imath(R)\subset (\Smaxv)^n$ is closed under
  the product order topology on $(\Smaxv)^n$. 

  (1)$\Rightarrow$(2). Suppose now that $\imath(U)$ is a signed elimination cone, and that $U$ is
  closed in $(\Smaxv)^n$. Then, using again~\Cref{lem-topcoincide},
  we get that $\imath(U)$ is closed
  in the Euclidean topology of $(\Rmax^n)^2$. Then, by
  \Cref{car-polar}, (1) $\Rightarrow$ (2), $\imath(U)$ is the signed part
  of a closed monotone pre-congruence $C$ of $\Rmax^n$. Then,
  by~\Cref{t:polar:new}, we have $C=A^\polarcouple$ where $A=C^\polarsingle$.
  Hence, $\imath(U)$ is the set of signed couples $(x^+,x^-) \in (\Rmax^n)^2$
  such that $\<x^+, a>\geq \<x^-,a>$ holds for all $a\in A$. By~\eqref{fromsigned-to-polar}, this means that $U$ is the set of vectors $x\in (\Smaxv)^n$
  such that $\<x,a>\succeq \zero$ for all $a\in A$, i.e., $U=A^\circ$.
  \end{proof}

The following examples show that the operation $\oplus_i$ can be considered as a tropical analogue of the Fourier--Motzkin elimination.

\begin{example}
\label{ex:FM1}
Suppose that a subset $A\subseteq \Rmax^3$ is contained in the following signed halfspaces: $a_1\geq a_2$ and $a_2\geq a_3$. This means that the vectors 
$(0,\ominus 0,\zero)^T$ and $(\zero, 0,\ominus 0)^T$ are contained in $A^{\circ}$. The corresponding vectors $x^1$, $x^2$ with components in $\Rmax^2$ and the vector $x^3=x^1\oplus_2 x^2$ are shown below:
\begin{equation*}
x^1=
\begin{pmatrix}
(0,-\infty)\\
(-\infty,0)\\
(-\infty,-\infty)
\end{pmatrix},\quad
x^2=
\begin{pmatrix}
(-\infty,-\infty)\\
(0,-\infty)\\
(-\infty, 0)
\end{pmatrix},\quad
x^3=
\begin{pmatrix}
(0,-\infty)\\
(-\infty,-\infty)\\
(-\infty, 0)
\end{pmatrix}\enspace .
\end{equation*}
The last vector corresponds to the vector $(0,\zero, \ominus 0)^T$ and to the signed halfspace $a_1\geq a_3$, which also contains $A$. In this way,
a transitivity property of the order relation $\leq$ is expressed
in terms of stability  under the operation $\oplus_2$.
\end{example}

\begin{example}
\label{ex:FM2}
Consider the following four vectors whose components are tropical signed numbers: $(1\;, \ominus 1,\; 3,\; \ominus 2)^T$,  $(2,\; \ominus 4,\; 1,\; \ominus 4)^T$, 
$(\ominus 3,\; 2,\; \ominus 4,\; 1)^T$ and $(\ominus 1,\; 3,\; 2,\; \ominus 3)^T$. They correspond to the following vectors with components in $\Rmax^2$:
\begin{equation*}
x^1=
\begin{pmatrix}
(1,-\infty)\\
(-\infty,1)\\
(3,-\infty)\\
(-\infty,2)
\end{pmatrix},\ 
x^2=
\begin{pmatrix}
(2,-\infty)\\
(-\infty,4)\\
(1,-\infty)\\
(-\infty,4)
\end{pmatrix}\quad
x^3=
\begin{pmatrix}
(-\infty,3)\\
(2,-\infty)\\
(-\infty,4)\\
(1,-\infty)
\end{pmatrix}
x^4=
\begin{pmatrix}
(-\infty,1)\\
(3,-\infty)\\
(2,-\infty)\\
(-\infty,3)
\end{pmatrix} \enspace .
\end{equation*}
We have four ways to eliminate the first component by means of $\oplus_1$ operation: 
$y^1=((-2) x^3 \oplus_1 x^1)^{\chech}$, $y^2=(x^4\oplus_1 x^1)^{\chech}$, $y^3=((-1)x^3\oplus_1 x^2 )^{\chech}$ and $y^4=(x^4 \oplus_1  (-1) x^2)^{\chech}$  and it can be checked that the resulting vectors with components in $\Rmax^2$ correspond to the following vectors with tropical signed components: 
\begin{equation*}
\begin{split}
y^1&=
\begin{pmatrix}
(-\infty,-\infty)\\
(-\infty, 1)\\
(3,-\infty)\\
(-\infty,2)
\end{pmatrix}\sim 
\begin{pmatrix}
\zero\\
\ominus 1\\
3\\
\ominus 2
\end{pmatrix},\quad
y^2=
\begin{pmatrix}
(-\infty,-\infty)\\
(3, -\infty)\\
(3,-\infty)\\
(-\infty,3)
\end{pmatrix}\sim 
\begin{pmatrix}
\zero\\
3\\
3\\
\ominus 3
\end{pmatrix}\\
y^3&=
\begin{pmatrix}
(-\infty,-\infty)\\
(-\infty, 4)\\
(-\infty,3)\\
(-\infty,4)
\end{pmatrix}\sim 
\begin{pmatrix}
\zero\\
\ominus 4\\
\ominus 3\\
\ominus 4
\end{pmatrix},\quad
y^4=
\begin{pmatrix}
(-\infty,-\infty)\\
(3, -\infty)\\
(2,-\infty)\\
(-\infty,3)
\end{pmatrix}\sim 
\begin{pmatrix}
\zero\\
3\\
2\\
\ominus 3
\end{pmatrix} \enspace .
\end{split}
\end{equation*}
In particular, the signed vectors
$(-2)x^3$ and $x^1$ encode the relations
$a_2 \oplus (-1)a_4\geq 1 a_1 \oplus 2a_3$
and $1a_1 \oplus 3 a_3 \geq 1a_2 \oplus 2a_4$,
respectively.
By adding $3a_3$ to the former relation, we deduce 
that $a_2 \oplus (-1)a_4 \oplus 3a_3 \geq 1a_1 \oplus 3a_3 \geq  1a_2 \oplus 2a_4$,
and after eliminating terms which cannot be active, we deduce
that $3a_3 \geq 1a_2 \oplus 2a_4$, which is precisely the inequality
expressed by the element $y^1=((-2)x^3 \oplus_1 x^1)^{\chech}$.
In that way, we see that the stability under the addition $\oplus_1$ allows
one to express the elimination of variable $a_1$ by a Fourier--Motzkin type
approach. See also Allamigeon et al.~\cite{uli2013} for an algorithmic discussion of tropical Fourier--Motzkin elimination.
\end{example}

\subsection{Valuation of polars}

In this subsection we relate signed polars to images by valuation of classical polars.
Let us denote by $\cl$ the closure of a set and by $\int$ its interior.
We use the product of order topologies, both on $(\Smaxv)^n$ and on $\puiseux^n$. 
Let us first observe the following useful property of tropical polars.

\begin{lemma}
\label{l:clint}
Let $A\subset \rmax^n$. Then $\cl\int (A^{\polar})=A^{\polar}$.
\end{lemma}
\begin{proof}
\newcommand{\xx}{x}
\newcommand{\zz}{z}
Let $\xx\in A^{\polar}$. Denote $N_+=\{i\colon \xx_i\succ\zero\}$, $N_-=\{i\colon \xx_i\prec\zero\}$ and 
$N_0=\{i\colon \xx_i=\zero\}$. Using the canonical representation $\xx=\xx^+\ominus \xx^-$ where both $\xx^+$ and $\xx^-$ are positive (in the sense of symmetrization) vectors with disjoint support, 
we define vectors $\zz^k\in(\Smaxv)^n$ for $k\geq 1$ as follows:

\begin{equation*}
\zz^k_i=
\begin{cases}
t_k+ \xx_i^+, &\text{if $i\in N_+$},\\
\ominus \xx_i^-, &\text{if $i\in N_-$},\\
M_k, &\text{ if $i\in N_0$},
\end{cases}
\end{equation*} 
where $M_k\succ\zero$ (positive in the sense of symmetrization) and $t_k>0$ (a real number, positive in the usual sense). Vectors $\zz^k$ can be chosen in such a way that 
$\lim_{k\to\infty} \zz^k=\xx$ in the product order topology. For example, we can choose $M_k=\log \epsilon_k$ where $\epsilon_k>0$ and $t_k$ satisfying  
$\epsilon_k=(e^{t_k}-1)\cdot \max_{i\in N_+} e^{\xx_i^+}$, with $\lim_{k\to\infty} \epsilon_k=0$. 
We then see that $e^{M_k}=\epsilon_k$ and $|e^{\zz^{k+}_i}-e^{\xx_i^+}|\leq\epsilon_k$, ensuring that $\lim_{k\to\infty} \zz^k=\xx$.



Now fix  $t'_k>0$ such that $t'_k<t_k$, and consider $y\in(\Smaxv)^{n}$ with canonical representation $y^+\ominus y^-$ 
that belongs to the open neighborhood of $\zz^k$ consisting of vectors $y$ such that
\begin{equation}
\label{e:neighborhood}
\begin{split}
& \text{for $i\in N_-$:}\quad   \zero\succ y_i \succ  \ominus (\xx_i^-+t'_k)\\
& \text{for $i\in N_+$:}\quad   y_i\succ \xx_i^++t'_k,\\
& \text{for $i\in N_0$:}\quad   y_i\succ \zero.
\end{split}
\end{equation}
Let  $a\in A$, then, 
using properties~\eqref{e:neighborhood}, we obtain
\[
\langle y^+,a\rangle \geq \bigoplus_{i\in N_0} y_i^+\odot a_i\oplus \bigoplus_{i\in N_+} y_i^+\odot a_i\geq 
\bigoplus_{i\in N_-} y_i^-\odot a_i=\langle y^-,a\rangle,
\]
as it can be argued that 
\[
\bigoplus_{i\in N_+} y_i^+\odot  a_i\geq t'_k+ \bigoplus_{i\in N_+} \xx_i^+\odot  a_i\geq t'_k +\bigoplus_{i\in N_-} \xx_i^-\odot  a_i\geq \bigoplus_{i\in N_-} y_i^-\odot  a_i.
\]
We thus obtain $y\in A^{\polar}$ for any such $y$, implying that $\zz^k\in\int(A^{\polar})$. As $\lim_{k\to\infty} \zz^k=\xx$, the 
claim follows.
\end{proof}

The following theorem is the main result of this subsection.

\begin{theorem}
  \label{t:polval}
  Let $\bm{A}$ be a closed semialgebraic subset of $\puiseux_{\geq 0}^n$. Then,
  \begin{align}
    \sval(\bm{A}^{\polar})= (\val \bm{A})^{\polar}.
        \label{l:easyincl}
\end{align}
\end{theorem}



The proof of~\Cref{t:polval} relies on auxiliary results.

\begin{lemma}
\label{l:tpolar}
Let $A$ be a subset of $\Rmax^n$ and let $x=x^+\ominus x^-$ belong to 
$\int A^{\polar}$. Then 
\begin{equation*} 
\langle x^+,a\rangle >\langle x^-,a\rangle,\quad\forall a\in A\backslash\{\zero\}.
\end{equation*}
\end{lemma}
\begin{proof}
By contradiction, suppose that there exists $a\in A\backslash\{\zero\}$ 
such that $\langle x^+,a\rangle =\langle x^-,a\rangle$.

Suppose first that the latter value is equal to $\zero$,
and let $i\in[n]$ be such that $a_i \neq \zero$.
Let $e_i$ denote the $i$th unit vector of $\Smax^n$, consider
$y= x\ominus t e_i$, with $t$ positive. Then, $y_j = x_j$ holds for all $j\neq i$
and $y_i = \ominus t$. It follows that
$\langle y^+,a\rangle=\zero< \langle y^-,a\rangle =  t a_i$.
This contradicts that $x$ belongs to $\int A^{\polar}$. 

Suppose now that $\langle x^+,a\rangle =\langle x^-,a\rangle \neq \zero$.
Then there exists $t\in\Rmax$ such that $t>\unit$ and $y=x^+\ominus t x^-$ still belongs to $A^{\polar}$. However, we have $\langle y^+,a\rangle<\langle y^-,a\rangle$ in contradiction 
with $y\in A^{\polar}$, and this proves the claim.
\end{proof}

We deduce the following.

\begin{lemma}
\label{l:intincl}
$\int((\val\bm{A})^{\polar})\subset \sval(\bm{A}^{\polar})$.
\end{lemma}
\begin{proof}
Take $x\in\int((\val\bm{A}))^{\polar})$ and consider $\bm{x}$ with coordinates defined by 
$\bm{x}_i=\sign(x_i)t^{|x_i|}$. Then for each $\bm{a}\in\bm{A}$ we have 
$\langle x^+,\val(\bm{a})\rangle>\langle x^-,\val(\bm{a})\rangle$ by \Cref{l:tpolar} and
therefore also $\langle \bm{x},\bm{a}\rangle\geq 0$ for the lift. This shows $x\in\sval(\bm{A}^{\polar})$.
\end{proof}

We are now ready to complete the proof of the main statement.

\begin{proof}[Proof of \Cref{t:polval}]

  We first prove the inclusion $\sval(\bm{A}^{\polar})\subset (\val\bm{A})^{\polar}.$ For each $x\in\sval(\bm{A}^{\polar})$ and $\bm{x}$ such that $x=\sval(\bm{x})$ we 
have that  $\langle \bm{x},\bm{a}\rangle\geq\bm{0}$ 
for all $\bm{a}\in\bm{A}$, which we can write as 
$\langle \bm{x^+},\bm{a}\rangle\geq\langle\bm{x^-},\bm{a}\rangle$. Taking valuation of this 
we obtain $\langle x^+,\val(\bm{a})\rangle\geq \langle x^-,\val(\bm{a})\rangle$ for each 
$a\in\bm{A}$, which shows the inclusion.

Now, combining \eqref{l:easyincl} and \Cref{l:intincl} we obtain
\begin{align}
\int((\val\bm{A})^{\polar})\subset \sval(\bm{A}^{\polar})\subset (\val\bm{A})^{\polar}\label{e-chain}
\end{align}
We recall that if $\mathbf{B}$ is a closed semi-algebraic subset of $\K^n$,
then, $\sval(\mathbf{B})$ is a closed subset of $(\Smaxv)^n$, this
follows from  Theorem 6.9 of~\cite{Jell2020} or alternatively
from of Corollary 4.11 of~\cite{AGS}.
Hence, applying the topological closure to the chain of
inclusions~\eqref{e-chain},  and using that: 1) $\cl\int((\val\bm{A})^{\polar})= (\val\bm{A})^{\polar}$ by \Cref{l:clint},
 2)  $\sval(\bm{A}^{\polar})$ is closed by the above property; 3) $(\val\bm{A})^{\polar}$ is trivially closed (being a tropical polar), we obtain the ``sandwich'' 
\begin{equation*}
(\val\bm{A})^{\polar}\subset \sval(\bm{A}^{\polar})\subset (\val\bm{A})^{\polar},
\end{equation*}
from which the claim follows.
\end{proof}

\begin{remark}
We know from~\cite{AGS} that if $\bfA$ is a closed
semi-algebraic subset of $\mathbb{K}_{\geq 0}^n$, then, the subset
of $\val (\bfA)$ with a given support is a closed semilinear
set. Moreover, this subset can be computed from a definition of $\bfA$
by polynomial inequalities, under a genericity condition, see~\cite[Coro~4.8]{AGS}.
\end{remark}

\section{Tropicalization of matrix cones and their polars}
\label{sec-completely}

\subsection{Tropical positive definite matrices and completely positive matrices}
\label{sec-tropsdp}

We say that a symmetric matrix $A\in (\Smax^\vee)^{n\times n}$ is a
{\em signed tropical positive semidefinite matrix}. 
if $x^T Ax \sgeq 0$ for all $x\in (\Smax^\vee)^n$

We denote by 
$\psd(\Smax)$
the set of {\em signed tropical positive semidefinite matrix}
and denote by
$\psd(\Rmax) = \psd(\Smax)\cap \Rmax^{n\times n}$
the set of {\em tropical positive semidefinite matrices}.
We first examine the special case of $2\times 2$ matrices. 
\begin{lemma}\label{sdp-dim2}
  Let $a,b,c\in \Smax^\vee$. Then,
 \[
 a x_1^2 \oplus bx_1x_2 \oplus cx_2^2 \sgeq \zero ,\; \forall x_1,x_2\in \Smax^\vee
 \]
  holds if and only if
  $a\sgeq 0,\; c\sgeq \zero$ and $b^2 \sleq ac.$
  \end{lemma}
\begin{proof}
  ``Only if''. By taking $x_1 =\unit$ and $x_2=\zero$, we see
  that the condition $a\sgeq \zero$ is necessary. Similarly, $c\sgeq \zero$
  is necessary. If $a=\zero$, then, by taking $x_2=\unit$ and $x_1 = \ominus ub$
  with $u\in \Rmax$ large enough, we get
  $a x_1^2 \oplus bx_1x_2 \oplus cx_2^2
  = \ominus ub^2 \oplus c\sgeq \zero$ which holds for all such $u$
  only if $b=\zero$. Then the condition $b^2\sleq ac$ trivially
  holds. By symmetry, the same is true if $c=\zero$.
  We may now assume that $a,c$ are positive elements
  of $\Smax^\vee$. Then, taking $x_1 =\unit/\sqrt{a}$
  and $x_2 = \epsilon\unit/\sqrt{c}$ with $\epsilon\in \{\unit,\ominus\unit\}$,
  we get $\unit \oplus \epsilon b/(\sqrt{ac})\sgeq \zero$, which
  is possible only if $|b|/\sqrt{ac}\sleq \unit$,
  or equivalently $b^2\sleq ac$. 

  ``If''. The case in which $a=\zero$ or $c=\zero$ is trivial,
  so, we suppose that $a$ and $c$ are both positive, and
 make the change of variable $x_1=y_1/\sqrt{a}$
  and $x_2=y_2/\sqrt{c}$. Then,
  $a x_1^2 \oplus bx_1x_2 \oplus cx_2^2\sgeq \zero$
  holds iff
  \begin{align}
y_1^2 \oplus u y_1y_2 \oplus y_2^2 \sgeq \zero\label{e-chvar}
\end{align}
where $u:= b/(\sqrt{a}\sqrt{c})$ is such that
$|u|\sleq \unit$. Without loss of generality,
we may assume that $y_1^2\sgeq y_2^2$. Then
\eqref{e-chvar} can be rewritten
as $y_1^2 \oplus uy_1y_2\sgeq\zero$, which
holds since $|uy_1y_2|= |u|\sqrt{|y_1|^2 |y_2|^2} \sleq |u||y_1|^2 \sleq y_1^2$.

\end{proof}
The condition $b^2\sleq ac$ in \Cref{sdp-dim2} is the tropical analogue of the nonpositivity of the discriminant.


    
\begin{theorem}\label{th-psdsmax}
  \begin{equation*}
  \begin{split}
  \psd(\Smax)
  =\{ A \in (\Smax^\vee)^{n\times n} \mid & A_{ii}\sgeq \zero, \forall i\in [n],\\
  & A_{ij}=A_{ji},\; A_{ij}^2 \sleq A_{ii}A_{jj}, \forall i,j\in [n], i\neq j\}\enspace .
  \end{split}
  \end{equation*}
In other words, $\psd(\Smax)$ consists of those matrices $A\in (\Smax^\vee)^{n\times n}$ such that $A^-_{ii}=\zero$ for all $i\in [n]$ and $(A_{ij}^\pm )^2 \sleq A_{ii}^+ A_{jj}^+$ for $i,j\in [n]$ with $i\neq j$.
\end{theorem}
\begin{proof}
  ``Only if''. This is obtained from~\Cref{sdp-dim2}, by noting
  that if $A\in (\Smax^\vee)^{n\times n} $, then any $2\times 2$ principal
  submatrix of $A$ must be in $\psdtwo(\Smax)$.

  ``If''. If $A$ satisfies the condition of the theorem, by \Cref{sdp-dim2},
  we have
  $x_{i}A_{ii}x_i \oplus x_j A_{jj}x_j \oplus x_i A_{ij}x_j\sgeq \zero$ for all $i\neq j$.
  By summing these inequalities over all $i\neq j$, we arrive
  at $x^{T}Ax\sgeq \zero$.
\end{proof}

We have the following immediate corollary.

\begin{corollary}\label{c-psdrmax}
  \[
  \psd(\Rmax)
  =\{ A \in \Rmax^{n\times n} \mid  \;A_{ij}=A_{ji},\; A_{ij}^2 \sleq A_{ii}A_{jj}, \forall i,j\in [n], i\neq j\}\enspace .
  \]
\end{corollary}

A matrix $X\in\rmax^{n\times n}$ is called {\em tropical completely positive of order $k$} (where $k\geq 1$) if it has entries
\[
X_{ij}=  \langle y^i,y^j\rangle \enspace,\quad i,j\in [n]\enspace,
\]
for some vectors $y^i\in \rmax^k$ (for $i\in [n]$), where $\langle x,y\rangle:=\bigoplus_i x_i \odot y_i$ denotes the canonical scalar product on $\rmax^k$.
The set of all such matrices $X$ for a fixed $k$ is denoted by $\cpk(\rmax)$.

A matrix $X\in\rmax^{n\times n}$ is called {\em tropical completely positive semidefinite of order $k$} (where $k\geq 1$) if it has entries
\[
X_{ij}= \langle Y^i,Y^j\rangle \enspace,\quad i,j\in [n]\enspace,
\]
for some $Y^i\in \psdk(\rmax)$ (for $i\in [n]$), where the Frobenius scalar
product $\langle \cdot, \cdot \rangle$ is understood in the tropical sense,
i.e., $\langle X, Y\rangle :=\bigoplus_{ij}X_{ij}\odot Y_{ij}$. The set of all such matrices $X$ for a fixed $k$ is denoted by $\cpsdk(\rmax)$.

The sets of tropical completely positive matrices and, respectively, tropical completely positive definite matrices are $\cp(\rmax)=\cup_{k\geq 1} \cpk(\rmax)$
and, respectively, $\cpsd(\rmax)=\cup_{k\geq 1}\cpsdk(\rmax)$. These sets are also referred to as {\em tropical completely positive cone} and {\em tropical completely positive semidefinite cone.} In what follows for a square matrix $X$ we denote by $\diag(X)$ the vector of diagonal entries of $X$.

\begin{lemma}\label{lem-diagonal}
  For all $Y,Z\in \psd(\rmax)$,
  \[
  \langle Y, Z \rangle = \langle \diag(Y),\diag(Z)\rangle\enspace .
  \]
\end{lemma}
\begin{proof}
  We have
  $\langle Y, Z \rangle=\bigoplus_{ij}Y_{ij}Z_{ij}
  \sleq \bigoplus_{ij}\sqrt{Y_{ii}Y_{jj}}\sqrt{Z_{ii}Z_{jj}}$.
  Observe that
\[ \sqrt{Y_{ii}Y_{jj}}\sqrt{Z_{ii}Z_{jj}}=
\sqrt{Y_{ii}Z_{ii}} \sqrt{Y_{jj}Z_{jj}}
  \sleq Y_{ii}Z_{ii} \oplus Y_{jj}Z_{jj}
  \sleq \langle \diag(Y),\diag(Z)\rangle\enspace,
  \]
  showing
  that $\langle Y,Z\rangle \sleq \langle \diag(Y),\diag(Z)\rangle$.
  The other inequality is trivial.
\end{proof}
\begin{remark}\label{rk-badpsd}
  If $Y,Z\in \psd(\smax)$, then (in general) we only get\\
  $\langle Y,Z \rangle \balance \langle \diag(Y),\diag(Z)\rangle $ and
  $|\langle Y,Z \rangle | $ $=| \langle \diag(Y),\diag(Z)\rangle |$. 
  \end{remark}
We note that the two classes of matrices $\cpsd$ and $\cp$
coincide in the tropical setting.
\begin{proposition}\label{pr-2}
  $\cpsdk(\rmax)=\cpk(\rmax)$.
\end{proposition}
\begin{proof}
  If $X\in \cpk(\rmax)$, then $X_{ij}=(y^i)^Ty^j$ for some
  vectors $y^i\in \rmax^k$, and so $X_{ij} =\langle \diag(y^i),\diag(y^j)\rangle
  \in \cpsdk(\rmax)$. Conversely, if $X\in \cpsdk(\rmax)$, we
  have $X_{ij}=\langle Y^i,Y^j \rangle $ for some $Y^i\in \psdk(\rmax)$,
  and so, by \Cref{lem-diagonal}, $X_{ij}=(y^i)^Ty^j$
  where $y^i=\diag(Y^i)$, which implies that $X\in \cpk(\rmax)$. 
\end{proof}
\begin{remark}
  Recall that the cp-rank of a $n\times n$ matrix $\bm{A}$ is the
  smallest integer $k$ such that $\bm{A}\in \cpk(\mathbb{K})$. The
  csdp-rank of $\bm{A}$ is the smallest integer $k$ such that $\bm{A}\in \csdpk(\mathbb{K})$. \Cref{pr-2} entails that the tropical analogues of these two
  notions of rank coincide.
\end{remark}

\begin{theorem} 
\begin{equation*}
\begin{split}
&\sval(\psd(\puiseux)) = \psd(\Smax),\\
&\val (\psd(\puiseux)\cap \nnmat) =
  \val(\psd(\puiseux)) = 
  \psd(\Rmax).
\end{split}
\end{equation*}
\end{theorem}
\begin{proof}
We prove the first of these equalities. The inclusion $\sval(\psd(\puiseux))\subseteq\psd(\Smax)$ is obvious, because
for each $\bm{A}\in\psd(\puiseux)$ we have $\bm{A}_{ii},\bm{A}_{jj}\geq 0$ and $\bm{A}_{ij}^2\leq \bm{A}_{ij}\bm{A}_{ji}$.

Let us prove the other inclusion. Take $A\in\psd(\Smax)$ and also define matrix $\bm{B}$ with entries
\begin{equation*}
\bm{b}_{ij}=
\begin{cases}
(n-1), &\text{if $i=j$},\\
1, &\text{if $i\neq j$}.
\end{cases}
\end{equation*}
Now let $\bm{C}$ be the lift of $A$, with entries 
$\bm{c}_{ij}=\epsilon_{ij} \bm{b}_{ij} t^{|A_{ij}|}$ and $\epsilon_{ij}=\sign(A_{ij})$. 
Then $\bm{C}\in\psd(\puiseux)$ by \cite{AGS}, Lemma 5.4. 

The proof of the other equalities is similar and uses $\bm{c}_{ij}= \bm{b}_{ij} t^{|A_{ij}|}$
\end{proof}
\begin{corollary}[{\cite{Yu15}}]
  The set $\val(\psd(\puiseux)\cap (\puiseux^*)^{n\times n})$ coincides
  with the set of matrices $A=(A_{ij})\in \R^{n\times n}$
    such that $2 A_{ij}\leq A_{ii}+A_{jj}$.
  \end{corollary}


Let us recall the following result of Cartwright and Chan.
\begin{theorem}[{\cite[Proposition 2 and Theorem 4]{CartChan12}}]
  \label{th-cc}
We have, for all $k\geq \max(n,\lfloor n^2/4\rfloor)$, 
\begin{equation*}
\cp(\rmax)=\cpk(\rmax) = \{M\in\Rmax^{n\times n}\mid M_{ij}=M_{ji},\  M_{ij}^{2}\leq M_{ii} M_{jj}\ \forall i,j\}\enspace .
\end{equation*}
\end{theorem}
Together with \Cref{c-psdrmax} and~\Cref{pr-2}, this implies the following result:
\begin{corollary}
\label{c:CC}
$\cp(\Rmax) = \psd(\Rmax) = \cpsd(\Rmax).$
\end{corollary}

We also make the following observation:
\begin{proposition}\label{prop-cpval}
$\val(\cp(\K))=\cp(\rmax).$
\end{proposition}
\begin{proof}
  Indeed, consider a tropical matrix of the form $X=Y Y^T$.
 Then we can take the monomial lift $\bm{Y}_{ij}=t^{Y_{ij}}$, and then 
 $\val(\bm{Y}\bm{Y}^T)=Y Y^T=X$ by the properties of valuation over $\puiseuxpos$.  This shows $\cp(\Rmax)\subseteq\val(\cp(\puiseux))$.
 Similarly, if $\bm{X}=\bm{Y}\bm{Y}^T\in \cp(\K)$, we deduce
 that $\val(\bm{X})=\val(\bm{Y})(\val(\bm{Y})^T)$, showing
 that $\val(\cp(\puiseux))\subseteq\cp(\Rmax)$. 
  \end{proof}
The following result shows that the tropical analogue of the hierarchy~\eqref{e:primal-incl} collapses.

\begin{theorem}
\label{t:collapse1}
For all  $k\geq \max(n,\lfloor n^2/4\rfloor)$, 
\begin{equation*}
\begin{split}
\val(\cpk(\puiseux))&=  \val(\cp(\puiseux)) 
  = \val( \csdpk(\puiseux))\\
  &= \val ( \csdp(\puiseux)) =  \val(\psd \cap \nnmat) \enspace .
\end{split}
\end{equation*}
\end{theorem}
\begin{proof} We use the following sequence of inclusions:
\begin{equation*}
\begin{split}
   \cpk(\puiseux) 
 & \subseteq \csdpk(\puiseux) 
  \subseteq \csdp(\puiseux)\\ 
& \subseteq \psd \cap \nnmat
	\subseteq \{\bm{X}\colon \bm{X}_{ij}\geq 0,\ \bm{X}_{ii}\bm{X}_{jj}\geq \bm{X}_{ij}^2\quad\forall i,j\}.
\end{split}
\end{equation*}
The last of these inclusions holds since the nonnegativity of all $2\times 2$ matrices is a necessary condition for 
a matrix to be positive semidefinite, and the rest of them follow from~\eqref{e:primal-incl}.  
Taking the valuation, we obtain
\begin{align*}
  \val(\cp(\puiseux)) 
 &\subseteq \val(\csdpk(\puiseux))
  \subseteq \val(\csdp(\puiseux))
  \subseteq \val(\psd \cap \nnmat)\\
  &\subseteq
  \val(\{\bm{X}\colon \bm{X}_{ij}\geq 0,\ \bm{X}_{ii}\bm{X}_{jj}\geq \bm{X}_{ij}^2\quad\forall i,j\})\\
  &\subseteq
  \{X\in \rmax^{n\times n}\colon {X}_{ii}{X}_{jj}\geq {X}_{ij}^2\quad\forall i,j\}.
\end{align*}
By \Cref{th-cc}, the set
$\{X\in \rmax^{n\times n}\colon {X}_{ii}{X}_{jj}\geq {X}_{ij}^2\quad\forall i,j\}$
is precisely $\cp(\rmax)$, which coincides
with $\cpk(\rmax)$ for all $k\geq \max(n,\lfloor n^2/4\rfloor )$.
Moreover, by \Cref{prop-cpval}, $\cp(\rmax)$
coincides with $\val(\cp(\puiseux))$, hence, all the inclusions
in the latter chain of inclusions must be equalities. \end{proof}




\begin{remark}
\label{r:DJL}
Note that given a matrix $\bm{X}\in \cp(\K)$, the smallest integer
$k$ such that $\bm{X}\in \cpk(\K)$ is bounded by $\binom{n+1}{2}$
(this result is proved by Hannah and Laffey \cite{HL83}  for matrices over $\mathbb{R}$,
the same result is true over any real closed field, and in particular
over $\mathbb{K}$). 
For the tropical completely positive cone there is a better bound 
$\max(n,\lfloor\frac{n^2}{4}\rfloor)$, proved in Cartwright and Chan~\cite{CartChan12}, Theorem 4.
This bound originated from the work of
 Drew, Johnson and Loewy \cite{DJL94} who conjectured that the CP-rank 
of any completely positive matrix over $\mathbb{R}$ is at most 
$\lfloor\frac{n^2}{4}\rfloor$. 
The bound was then shown to be generally false in the usual mathematics,
in fact, Bomze, Schachinger and Ullrich~\cite{Bomze2015} showed that $\cpk(\K)\neq \cp(\K)$
for $k=n^2/2 +O(n^{3/2})$. 
However, the bound $\lfloor\frac{n^2}{4}\rfloor$ is true in tropical mathematics.
\end{remark}
\begin{remark}\label{rk-cpsd}
  As pointed out in Prakash et al.~\cite{Prakash2017}, there is no general upper bound
  for the cpsd rank of a $n\times n$ cpsd matrix over a real closed field.
  It shown there that this cpsd rank may be as high as $2^{\Omega(\sqrt{n})}$. Hence,
    $\csdpk(\puiseux)\neq \csdp(\puiseux)$ for $k=2^{O(\sqrt{n})}$,
    whereas~\Cref{t:collapse1} shows that the images of the two sets by the nonarchimedean valuation coincide.
  \end{remark}

\subsection{Tropical copositive and co-completely positive semidefinite matrices}


Let us now observe that the {\em tropical copositive cone}, which we
define as the tropical polar $\cp(\rmax)^{\polar}$ of the 
tropical completely positive cone, admits the same description as in the usual mathematics.

\begin{lemma}
We have
\begin{equation*}
\cp(\rmax)^{\polar}=\{A\in(\Smaxv)^{n\times n}\mid x^TAx\succeq \zero\ \text{for all $x\in\Rmax^n$}\}.
\end{equation*}
\end{lemma}
\begin{proof}
For any $A\in\cp(\rmax)^{\polar}$ and any $x\in\Rmax^n$ we have $\langle A,xx^T\rangle\succeq\zero$, which is 
the same as $x^TAx\succeq\zero$. 

If $A\in(\Smaxv)^{n\times n}$ satisfies this condition, then $\langle A,XX^T\rangle\succeq\zero$
for any $X\in\Rmax^{n\times k}$ since $XX^T=\bigoplus_{i=1}^k x^i (x^i)^T$ for the columns $x^i$ of $X$. Thus 
$A\in\cp(\rmax)^{\polar}.$
\end{proof}

We give a direct description of the same cone.

\begin{theorem}
\label{t:cocptrop}
We have $A\in\cp(\rmax)^{\polar}$ if and only if 
\begin{eqnarray*}
A_{ii} \succeq \zero \quad \forall i\in[n],\\
(A_{ij}^-)^2\preceq A_{ii}A_{jj}\quad \forall i,j\in[n] \enspace .  
\end{eqnarray*}
\end{theorem}
\begin{proof}
  Arguing as in the proof of \Cref{th-psdsmax}, it suffices
  to establish the equivalence in the $2\times 2$ case,
  i.e., to show that, for all $a,b,c\in \Smax^\vee$,
  \[
  ax_1^2 \oplus bx_1x_2 \oplus cx_2^2 \succeq \zero,\qquad \forall x_1,x_2\succeq \zero
  \]
  holds if and only if
  \[
  a,c\succeq\zero \text{ and } (b^-)^2 \preceq ac  \enspace .
  \]
  The ``Only if'' part is established as in \Cref{sdp-dim2}.
  Arguing as in the proof of this lemma, the only nontrivial case
  is when $a,c$ are positive, and then, instead
  of $\unit \oplus \epsilon b/(\sqrt{a}\sqrt{c}) \succeq \zero$
  for all signs $\epsilon\in \{\unit,\ominus \unit\}$, we only deduce
  that $\unit \oplus b/(\sqrt{a}\sqrt{c}) \succeq \zero$.
This gives $b^-\preceq \sqrt{a}\sqrt{c}$, i.e., $(b^-)^2\preceq ac$.
  The ``If'' part is established as in \Cref{sdp-dim2}.
\end{proof}

We now characterize $\cp(\rmax)^{\polar}$ in terms of the image
of the classical cone of copositive matrices $\cp^{\polar}(\puiseux)$
by the signed valuation.

 \begin{lemma}
\label{l:cocpval}
$\sval(\cp^{\polar}(\puiseux))=(\val(\cp(\puiseux))^{\polar} =\cp(\rmax)^{\polar}$.
\end{lemma}
\begin{proof}
By \Cref{prop-cpval} we have $\cp(\rmax)=\val(\cp(\puiseux))$, so its tropical polar is $\cp(\rmax)^{\polar}$ is the same as $(\val(\cp(\puiseux))^{\polar}$. The rest of the claim then immediately follows from \Cref{t:polval}.
\end{proof}
\begin{theorem}
\label{t:collapse2} 
For all $k\geq \max(n,\lfloor n^2/4\rfloor)$
\begin{equation*}
\begin{split}
\sval(\psd(\puiseux) + \nnmat)&=\sval( \csdp^\polar(\puiseux))=\sval(\csdpk^\polar(\puiseux))\\
&=\sval(\cpk^{\polar}(\puiseux))
=\sval(\cp^{\polar}(\puiseux))=\cp(\rmax)^{\polar}\enspace .
\end{split}
\end{equation*}
\end{theorem}
\begin{proof}
\Cref{t:collapse1} implies
\begin{equation*}
\begin{split}
(\val(\cpk(\puiseux)))^{\polar}&=  (\val(\cp(\puiseux)))^{\polar} 
  = (\val( \csdpk(\puiseux)))^{\polar}\\ 
  &=(\val ( \csdp(\puiseux)))^{\polar} =  (\val(\psd(\puiseux) \cap \nnmat)^{\polar} \enspace .
  \end{split}
\end{equation*}
for $k\geq \max(n,\lfloor n^2/4\rfloor)$. By \Cref{l:cocpval},  $(\val(\cp(\puiseux)))^{\polar}=\cp(\rmax)^{\polar}$, so all of these polars coincide with $\cp(\rmax)^{\polar}$ and are stable under the operator $\cl\int(\cdot)$. Therefore we can apply \Cref{t:polval}, which transforms this chain of equalities 
to the one that is claimed.


\end{proof}

\begin{remark}
The inclusion 
$\sval(\psd)\subset \sval(\psd+\nnmat)$ is strict,
since for example the matrix
$$
A=
\begin{pmatrix}
2 & 3 \\ 3& 2
\end{pmatrix}
$$
does not arise as the image of an element $\bfA$ of $\psd$ by the signed valuation (because the determinant of any such matrix $\bfA$ should have signed valuation $\ominus 6$, and therefore, this determinant would be negative).
\end{remark}

\section{Optimization in the Signed Tropical World}
\label{s:optimization}
\subsection{Optimization problems over the symmetrized tropical semiring}
We shall consider optimization problem over $\Smax$ of the form:

\begin{equation*}
\begin{split}
\inf\ f(x)\\
\text{s.t.}\ x\in A\subseteq(\Smax^\vee)^n,\\
f(x)\in\Smax^{\vee}, 
\end{split}
\end{equation*}
where the $\inf$ is taken with respect to the $\succeq$ relation.
It is important to require that $f(x)\in\Smax^{\vee}$, for otherwise $\succeq$ is not even an order relation
(the transitivity breaks). The assumption that $A\subset(\Smax^{\vee})^n$ is relevant in applications, for as observed above, signed tropical numbers represent images of elements of ordered valued fields by the signed valuation.

\begin{example}
\label{ex:pol}
{\rm Consider the following problem 
\begin{equation*}
\begin{split}
& \inf 4x^2\oplus 4x\oplus 0,\\
\text{s.t.}\  & x\in \Smax^{\vee}, 4x^2\oplus 4x\oplus 0\in\Smax^{\vee}.
\end{split}
\end{equation*}
For $f(x)=4x^2\oplus 4x\oplus 0$ and $x\in\Smax^{\vee}$ we have, in terms of the usual arithmetics, that
\begin{equation*}
f(x)=
\begin{cases}
4+2x, & \text{if $x\succeq 0$ or $x\prec \ominus 0$,}\\
4+x, &\text{ if $-4\preceq x\preceq 0$,}\\
0, &\text{if $\ominus -4\prec x\preceq -4$.}\\
\ominus(4+x), &\text{if $\ominus 0\prec  x\prec \ominus -4$.}
\end{cases}
\end{equation*}
Note that when $x=\ominus -4$ and $\ominus 0$ then $f(x)$ is balanced,
so these values of $x$ are excluded from optimization. The minimum of $f(x)$ is
$\ominus 4$ and it is obtained as $x$ tends to $\ominus 0$ from the left, see \Cref{f:poly}}.
\end{example}

\begin{figure}[h!]
\centering
\includegraphics[scale=1]{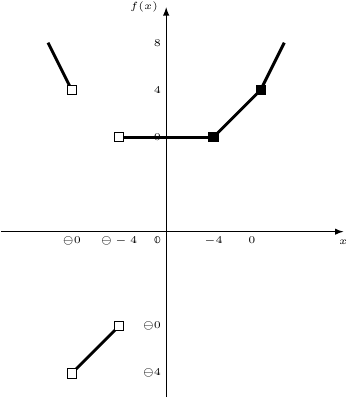}
\caption{Polynomial $f(x)$ over the symmetrized semiring\label{f:poly}}
\end{figure}

\if{
\begin{figure}[h!]
\centering
\begin{tikzpicture}[scale=0.4]\tiny
\tikzstyle{smallpoint} = [rectangle, fill=black, draw, minimum size=1mm]
\tikzstyle{voidsmallpoint}=[rectangle, fill=white, draw, minimum size=1mm]
   \tkzInit[xmax=7,ymax=9,xmin=-7,ymin=-7]
   \tkzSetUpAxis[ticka=0pt, tickb=0pt]    
 \tkzDrawXY[/tkzdrawX/label=$x$,/tkzdrawY/label=$f(x)$]

\tkzDefPoint(-5,8){leftpoint}
\tkzDefPoint(-4,6){corner1up}  
\tkzDefPoint(-4,-6){corner1down} 
\tkzDefPoint(-2,-4){corner2down}
\tkzDefPoint(-2,4){corner2up}
\tkzDefPoint(2,4){corner3}
\tkzDefPoint(4,6){corner4}
\tkzDefPoint(5,8){rightpoint}

\draw[line width=0.5mm] (leftpoint) -- (corner1up);
\draw[line width=0.5mm] (corner1down) -- (corner2down);
\draw[line width=0.5mm] (corner2up) -- (corner3);
\draw[line width=0.5mm] (corner3) -- (corner4);
\draw[line width=0.5mm] (corner4) -- (rightpoint);

\node at (-4,6)[voidsmallpoint] (c1) {};  
\node at (-4,-6)[voidsmallpoint] (c2){};
\node at (-2, -4) [voidsmallpoint] (c3){};
\node at (-2, 4) [voidsmallpoint] (c4){};
\node at (2, 4) [smallpoint] (c5){};
\node at (4, 6) [smallpoint] (c6){};

\coordinate [label=below left: $\zero$] (xx1) at (0,0);
\coordinate [label=below: $-4$] (xx2) at (2,0);
\coordinate [label=below left: $0$] (xx3) at (4,0);
\coordinate [label=below: $\ominus -4$] (xx2) at (-2,0);
\coordinate [label=below: $\ominus 0$] (xx3) at (-4,0);

\coordinate [label=left: $8$] (yy1) at (0,8);
\coordinate [label=left: $4$] (yy2) at (0,6);
\coordinate [label=left: $0 $] (yy3) at (0,4);
\coordinate [label=left: $\ominus 0$] (yy4) at (0,-4);
\coordinate [label=left: $\ominus 4$] (yy5) at (0,-6);
\end{tikzpicture}
\caption{Polynomial $f(x)$ over the symmetrized semiring\label{f:poly}}
\end{figure}
}\fi

Let $f =\bigoplus_{k=0}^n a_k x^k \in\Smax^\vee[x]$ be a formal univariate polynomial of degree $n$ (i.e., $a_n\neq \zero$), with signed coefficients.
We denote by $x\mapsto f(x)$ the associated polynomial function from
$\smax^\vee$ to $\smax$.
The {\em roots} of $f$ are defined as the points $x\in \Smax^\vee$ such that
$f(x) \balance \zero$.
  The moduli of the roots of $f$ are necessarily
 points of nondifferentiability of the map $x\mapsto |f|(x)$; however, not all nondifferentiability points yield roots (Baccelli et al.~\cite{BCOQ}).
  In fact, the notion of multiplicity of root in the setting of semirings or hyperfields with signs is a delicate one, see Baker and Lorscheid~\cite{baker2018descartes}, Gunn~\cite{gunn}.

  We consider the optimization problem
  \begin{align}\label{e-univariate}
  \inf f(x) , \qquad x\in\Smax^\vee, f(x) \in \Smax^\vee
  \end{align}
  and let us define the formal polynomial $|f|\in \Rmax[x]$ by
$|f|(x) = \bigoplus_{k=0}^n |a_k | x^k.$
 Recall that $\smax^\vee$ is equipped with the order topology.
  For all $\alpha\in \smax^\vee$, we denote by
  $f(\alpha^-)=\lim_{\beta \to \alpha, \beta\prec\alpha} f(\beta)$
  and similarly $f(\alpha^+)=\lim_{\beta \to \alpha, \beta\succ\alpha} f(\beta)$
  The following proposition shows that optimization
  problems for univariate tropical polynomials admit an explicit solution.
\begin{proposition}
\label{p:univariate}
The following claims hold for~\eqref{e-univariate}:
  \begin{enumerate}[label={\rm (\roman*)}]
  \item If the degree $n$ is odd, the objective function of the optimization problem~\eqref{e-univariate} is unbounded from below.
  \item The same is true if the degree $n$ is even, and if $a_n$ is negative.
  \item If the degree $n$ is even, if $a_n$ is positive, if $f$ has
    at least one root, then the value of the infimum is
    $\ominus |f|(|\alpha|)$, where $\alpha$ is any root of maximal modulus. Moreover,
    $\ominus |f|(|\alpha|) = f(\alpha^-)$ if $\alpha$ is positive, and 
    $\ominus |f|(|\alpha|) = f(\alpha^+)$ if $\alpha$ is negative.
    \item If the degree of $n$ is even, if $a_n$ is positive,
      and if $f$ has no root, then the value of the infimum is $f(\zero)$.
    \end{enumerate}
\end{proposition}
\begin{proof} 
We first observe that 1) the map $y\mapsto |f(y)|$ from $\rmax$ to $\smax^\vee$ is increasing, 2) $f(x)=a_nx^n$ when 
\begin{equation}
\label{e:|x|}
|x|\geq \bigoplus_{k=1}^n (|a_{n-k}|(|a_n|)^{-1})^{1/k}\enspace .
\end{equation}
We next prove all claims of the proposition.

(i),(ii): We observed that $f(x)$ equals $a_nx^n$ for positive or negative $x\in\smax^{\chech}$ satisfying \eqref{e:|x|}.
Let $n$ be odd.  If $a_n\succ\zero$ then $f(x)=a_nx^n$ is unbounded from below as $x\to\bot$, and if $a_n\prec\zero$ then the value $f(x)=a_nx^n$ is unbounded from below as $x\to\top$. If $n$ is even and $a_n\prec\zero$, then $f(x)=a_nx^n$ is unbounded from below as $x\to\bot$ and $x\to\top$.
  
 (iii):  In this case $f(x)$, being equal to $a_nx^n$, stays positive for all positive or negative $x\in\smax^{\chech}$ satisfying \eqref{e:|x|}. 
 If $\alpha$ is the root of $f$ over $\Smax$ that has a maximal modulus and is positive, then $f(x)$ remains positive for all $x\succ \alpha$, and $f(x)$ changes sign
  at point $x=\alpha$. Considering any increasing sequence $\alpha_k\in \Smax^\vee$ converging
  to $\alpha$, we obtain that $\inf_k f(\alpha_k)= \ominus |f|(|\alpha|)$,
  and using the monotonicity property of $y\mapsto |f(y)|$, we get
  that the infimum of $f$ over $\smax^\vee$ is equal to $\ominus |f|(|\alpha|)$.\\
If $\alpha$ is the root of $f$ over $\Smax$ that has a maximal modulus and is negative, then $f(x)$ remains positive for all $x\prec \alpha$, and $f(x)$ changes sign
  at point $x=\alpha$. Considering any decreasing sequence $\alpha_k\in \Smax^\vee$ converging
  to $\alpha$, we obtain that $\inf_k f(\alpha_k)= \ominus |f|(|\alpha|)$, and hence 
  the infimum of $f$ over $\smax^\vee$ is equal to $\ominus |f|(|\alpha|)$  (using the monotonicity property of $y\mapsto |f(y)|$).
 
 (iv): When $f$ has no root, it follows from the monotonicity of $y\mapsto |f(y)|$  
 that the infimum is equal to $f(\zero)$. 
\end{proof}





\subsection{Tropical quadratic programming}
We now consider tropical quadratic programming problems. Given
quadratic forms $f_i(x) = x^T A_i x \oplus b_i^T x \oplus c_i$ where
$A_i\in (\smax^\vee)^{n\times n}$, $b_i\in (\smax^\vee)^{n}$, and
$c_i\in \smax^\vee$, for $i\in \{0,\dots,m\}$,
\begin{align}
  \inf_{x\in (\smax^\vee)^n} f_0(x) ; \qquad f_i(x)\preceq \zero, \qquad i\in [m]
  \enspace .
  \label{def-qopt}
  \end{align}
Recall that in our definition of an optimization problem, the infimum
is taken over all $x$ such that $f_0(x)$ is signed. 
The tropical quadratic feasibility problem consists
in checking whether the inequalities  $f_i(x)\preceq \zero$ for all $i\in [m]$
hold for some $x\in (\smax^\vee)^n$. We note that the two inequalities
$f_i(x)\preceq \zero$ and $f_i(x)\succeq \zero$ are equivalent
to $f_i(x)\balance \zero$, so ``equalities'' involving
quadratic functions can be coded in this way.

The following observation allows us to formulate tropical copositivity
in terms of quadratic optimization.
\begin{proposition}
 A matrix $A\in(\smax^\vee)^{n\times n}$ is copositive iff
the value of the  quadratic optimization problem
$\inf_{x\in (\smax^\vee)^n}\{ x^T Ax,\  x\succeq \zero\}$ 
is $\zero$.
\end{proposition}
\begin{proof}
Indeed, if $A$ is copositive, we have by definition
$x^T Ax \succeq \zero$ for all $x\succeq \zero$. Moreover,
by taking $z=\zero$, we get $x^\top Ax=\zero\in \smax^\vee$,
showing that the above infimum is equal to $\zero$.
Conversely, suppose that $A$ is not copositive. Then,
there exists a vector $x\in (\smax^\vee)^n$
such that $x^TAx\prec \zero$ and $x^TAx$ is signed.
Considering $ux$ instead of $x$, with $u$ positive,
it follows that the value of the above infimum
is bounded above by $u^2(x^TAx)$, for all such $u$, and so,
the above infimum is equal to $\bot$.
\end{proof}
Testing whether a matrix is copositive is a co-NP complete problem,
see Dickinson and Gijben~\cite{Dickinson2014}. In contrast, it follows from \Cref{t:cocptrop} that
checking whether a tropical matrix is copositive can be done in polynomial time.
I.e., an important subclass of quadratic optimization problems which is
NP-hard in the classical case becomes polynomial in the tropical
case.  However, the following result shows that, in general, tropical
quadratic optimization remains NP-hard.


\begin{theorem}\label{th-nphard}
  The tropical quadratic feasibility problem is NP-hard.
\end{theorem}
\begin{proof}
  Let $0$ and $1$ denote the usual zero and unit elements
  of $\mathbb{R}\subset \rmax\subset \smax$, not to be confused
  with the $\zero$ and $\unit$ elements of $\smax$. The relations
  $s\in \smax^\vee$ and $s^2\ominus 1 s\oplus 1\balance \zero$
  are equivalent to $s\in \{0,1\}$ (to see this, note that
  $s^2\ominus 1 s\oplus 1=  (s\ominus 0)(s\ominus 1)$.
  Then,  if $x_i,y_i\in\{0,1\}$, the relation $x_iy_i \balance 1$ for $i\in [m]$  
  is equivalent to 
  ``$x_i=0$ iff $y_i=1$''. In that way, the solutions of these relations
  encode $m$ Boolean variables together with the negated variables.
  Consider now an instance of 3-SAT, in which the Boolean variables
  are $x_1,\dots,x_m$. A clause $C_i$ can be written
  as $E \vee F \vee G$ where each term $E,F,G$ is of the form
  $x_j$ or $\neg x_j$ for some $j$. Let us interpret
  such a clause in $\smax$, replacing $\neg x_j$ by $y_j$,
  and replacing $\vee$ by $\oplus$.
  E.g., the validity of the clause $x_3\vee \neg x_7\vee x_9$
  is coded by the linear constraint $ x_3 \oplus y_7 \oplus x_9 \balance 1$. Recall also that a relation of the form $u\balance v$ is equivalent
  to the conjunction of the two inequalities $u\preceq v$ and $u\succeq v$.
In that way, we express an instance of 3-SAT as a collection
  of inequalities in the variables $x_1,\dots,x_m,y_m,\dots,y_m$.
  It follows that the tropical quadratic feasibility problem is NP-hard.
\end{proof}
 In view of the definition of signed tropically positive semidefinite matrices, it is natural to introduce also the class 
 of tropical positive definite matrices: 
\begin{equation*}
  \begin{split}
  \pd(\Smax)
  =&\{ A \in (\Smax^\vee)^{n\times n} \mid A_{ii}\succ \zero, \forall i\in [n],\\
   &A_{ij}=A_{ji},\; A_{ij}^2 \prec A_{ii}A_{jj}, \forall i,j\in [n], i\neq j\}.
   \end{split}
 \end{equation*}
Let us now consider the following optimization problem with positive definite matrix $A$:
\begin{equation}
\label{e:quadropt}
\inf_{x\in (\Smax^\vee)^n,\;  x^TAx \oplus b^Tx\in \Smax^\vee} x^TAx \oplus b^Tx\in \Smax^\vee \enspace ,
\end{equation}
We will also consider the following optimization problem:
\begin{equation}
\label{e:quadroptlift}
\min_{\bfx\in\mathbb{K}^n} \bfx^T\bfA\bfx +\bfb^T\bfx \enspace,
\end{equation}
where $\bfA=\bfA^T \in \mathbb{K}^{n\times n}$ and $\bfb\in\mathbb{K}^n$.

We now define the {\em tropical comatrix}
$\com A$ to be the matrix whose $(i,j)$ entry equals $(\ominus \unit)^{i+j}
\det A(i,j)$ where $A(i,j)$ is obtained by deleting row $i$ and column $j$ of
$A$, and $\det $ denotes the determinant of a matrix with entries
in $\smax$, which is defined by the usual formula, see e.g.~\cite{BCOQ,guterman} for background.
If $A=\sval(\bfA)$ with $\bfA\in \K^{n\times n}$, and $A$ is
generic, then,  for all $i,j\in [n]$, there is only one term of maximal
modulus in the expansion of $\det A(i,j)$ as a signed tropical sum of weights of $(n-1)!$ permutations, and it follows that $\com A =\sval \com\bfA$ where $\com\bfA$ denotes the usual
comatrix of $\bfA$. 

\begin{proposition}
\label{p:quadropt}
Suppose that $A\in (\Smax^\vee)^{n\times n}$ is positive definite, let $b\in (\Smax^\vee)^n$, and
let $\bfA=\bfA^T\in \K^{n\times n}$ and $\bfb\in\K^n$ be such that
$A=\sval \bfA$ and $b=\sval \bfb$. Then the following claims hold for~\eqref{e:quadropt} and \eqref{e:quadroptlift}:
\begin{enumerate}[label={\rm (\roman*)}]
 \item The optimal value of~\eqref{e:quadropt} is equal to $\ominus b^T \diag(A)^{-1} b$ and is obtained by 
considering a minimizing sequence converging
to $\bar{x}:= \ominus \operatorname{diag}(A)^{-1} b$.
\item The optimal value of~\eqref{e:quadropt} coincides with the image by the signed nonarchimedean valuation
  of the optimal value of~\eqref{e:quadroptlift}, equal to $-\bfb^T\bfA^{-1}\bfb/4$.
\item If the entries of $A$ and $b$ are generic, then, the image by the signed nonarchimedean valuation of the unique optimal solution $\bfx^*=-\bfA^{-1}\bfb/2$ of~\eqref{e:quadroptlift}
  is given by $ x^*= \ominus (\det A)^{-1} (\com A)^T b$.
\end{enumerate}
\end{proposition}
\begin{proof}
(i): Since we have, for all $x\in(\Smax^\vee)^n$ such that $x_i\neq\zero$ and $x_j\neq \zero$,
\[
|x_iA_{ij}x_j|<|x_i|A_{ii}^{1/2}A_{jj}^{1/2}|x_j|\leq\max(x_i^2A_{ii},x_j^2A_{jj}), 
\]
it is easy to see that $x^TAx=x^T\operatorname{diag}(A)x$ holds
for all $x\in(\Smax^\vee)^n$. Using this we obtain that
\begin{equation*}
x^TAx\oplus b^Tx=\bigoplus_{i=1}^n A_{ii}x_i^2\oplus b_ix_i\enspace .
\end{equation*}
Minimizing $a_{ii}x_i^2\oplus b_ix_i$ for each $i$ we find (say, using Proposition \ref{p:univariate}) that since the only root is 
$\ominus A_{ii}^{-1}b_i$, the optimal value is attained by a minimizing sequence converging to this root, and the optimal value is
$\ominus b_i A_{ii}^{-1}b_i$. The optimal value of~\eqref{e:quadropt} is then obtained by summing up these values, thus it is
$\ominus b^T\operatorname{diag}(A)^{-1}b$, and it is attained by a minimizing sequence of vectors converging to 
$\ominus\operatorname{diag}(A)^{-1}b$.   

(ii): We make the change of variables
$x=Dx'$, where $D$ is the diagonal matrix $D$ with entries
  $D_{ii}=A_{ii}^{-1/2}$, which amounts to replacing $A$ with $A'=DAD$ and 
  $b$ with $b'=Db$ in \eqref{e:quadropt}. We then have $A'_{ii}=\unit$ for all $i$
and $A'_{ii}A'_{jj}=\unit > (A'_{ij})^2$ for all $i\neq j$, implying that $\det A' =\unit$.
Then, we write $A'=I\ominus C$, where $C$ is a matrix with diagonal entries
equal to $\zero$, and apply
the last statement of Akian, Gaubert and Niv~\cite{adi}, Theorem~2.39,
which entails that
\begin{align}\label{e-comastar}
  (\com A')^T  = C^*\coloneqq I \oplus C \oplus \dots\oplus  C^{n-1}
  \enspace .
  \end{align}
  We also make the change of variables $\bfx = \mathbf{D} \bfx'$, where $\mathbf{D}\in \K^{n\times n}$
  is the diagonal matrix with entries $\mathbf{D}_{ii}=t^{D_{ii}}$, and set $\bfA'\coloneqq
  \mathbf{D}\bfA \mathbf{D}$ and $\bfb'\coloneqq \mathbf{D}\mathbf b$.
  The value of the ``lifted'' problem~\eqref{e:quadroptlift} is given
  by
  \begin{align}
  \mathbf{v}= -\frac{\det(\bfA')^{-1}}{4} \bfb'^T (\com \bfA')^T \bfb' = -\frac{\det(\bfA')^{-1}}{4} \sum_{i,j\in[n]} \bfb'_i (\com\bfA')_{ji}
  \bfb'_j \enspace .\label{e-sumexpand}
  \end{align}
  Since $C_{ij}<\unit $ for all $i\neq j$, we deduce from~\eqref{e-comastar} that, for $i\neq j$,
  \[
  \val (\com \bfA')^T_{ij} \leq |C_{ij}^*|<\unit \enspace .
  \]
  Moreover, $\sval(\com\bfA')^T_{ii}=A'_{ii}=\unit$. Then, for $i\neq j$,
  $\val(\bfb'_i (\com\bfA')_{ji}\bfb'_j )< \val({\bfb'}_i^2) \oplus \val({\bfb'}_j^2)
  = \val(\bfb'_i (\com \bfA')^T_{ii}\bfb'_i) \oplus \val(\bfb'_j (\com \bfA')^T_{jj}\bfb'_j)$.
  Hence,
  in the sum in~\eqref{e-sumexpand},
  only the diagonal terms, obtained when $i=j$, have a maximal valuation. We also observe that
  the sign of each of these terms
  is positive. Moreover, $\sval\det(\bfA')=\unit$. It follows
  that 
  $\sval \mathbf{v} =   \ominus b'^T (\det A')^{-1}(\com A')^T b'$. Using~\eqref{e-comastar} and that $\det A' =\unit$ this becomes 
  \[
  \ominus\bigoplus_{i,j\in [n]} b'_{i}C^*_{ij}b'_j=
  \ominus \bigoplus_{i\in[n]} b_iA_{ii}^{-1}b_i =\ominus b^T \diag(A)^{-1}b,
  \] 
  as claimed.

(iii):  
  If the entries of $A$ are generic, we noted above that
  the matrix $\com A = \sval \com \bfA$ is signed.
  If in addition the entries of $b$ are generic, then the entries of ${\com A}^Tb$ are also signed and ${\com A}^Tb = \sval {\com \bfA}^T \bfb$.
  We finally claim that, since $A$ is positive definite,
  we have
  \begin{align}
    \label{showdet}
    \det A = \bigodot_{i\in [n]}A_{ii}
    =\sval \det\bfA\enspace .
  \end{align}
  Indeed, after making the change of variables of (ii),
we may assume that $A_{ii}=\unit$ for all $i$,
  with $A_{ii}A_{jj}=\unit > A_{ij}^2$. Considering the expansion
  of $\det A$, we see that the identity permutation yields
  the unique term with maximal modulus, from which~\eqref{showdet}
  readily follows. These facts imply that $x^*$ is the image by signed valuation of 
  $\bfx^*$.
\end{proof}
\begin{remark}
  The valuation of the optimal solution $x^*=\sval \bfx^*$ may differ
  from the vector $\bar{x} = \ominus \operatorname{diag}(A)^{-1} b$
  defined in (i). Indeed, consider
  \[
  A = \left(\begin{array}{cc} 0 & \ominus (-1)\\
    \ominus (-1) & 0 \end{array}\right),
  \quad \bfA = 
   \left(\begin{array}{cc} 1 & - t^{-1}\\
     - t^{-1} & 1 \end{array}\right),
   \quad
   b = \left(\begin{array}{c} 0\\ \theta\end{array}\right),
\quad\bfb= 
   \left(\begin{array}{c} 1\\ t^{\theta}
\end{array}\right),
   \]
   with $\theta\in\R\subset \Smax^\vee$,  so that $A=\sval \bfA$
   and $b=\sval \bfb$. We have $\bfx^*= -(1-t^{-2})^{-1}(1 + t^{-1} t^{\theta},
   t^\theta + t^{-1})^T$, $x^*=\ominus (\unit \oplus (-1 +\theta), \theta \oplus (-1))^T$,
   whereas $\bar{x}=\ominus (0,\theta)^T$. We observe that for $\theta>1$,
   $x^*=\ominus((-1+\theta), \theta)^T \neq \bar{x}$.
\end{remark}
\section{Concluding remarks}

In this paper we have characterized the polars of subsets $A$ of $\Rmax^n$ in terms of signed elimination cones (cones of signed vectors, stable under a special addition). Upon introducing and characterizing the tropical analogues of positive semidefinite matrices, completely positive semidefinite matrices and copositive matrices  we showed that both the hierarchy of classical matrix cones and the hierarchy of their polars collapse under signed tropicalization. We also studied some optimization problems over symmetrized tropical semiring. 
Let us now also mention the following two future research directions arising from this research.

Firstly, instead of defining the tropical polar of a subset $A$ of $\Rmax^n$ we may more generally consider a subset $A$ of $(\Smax^\vee)^n$
and still define its signed polar as in~\eqref{e-def-signedpolar}. This leads
to a larger class of signed polars, and it would be interesting
to extend the present results to this class.


Secondly, in Section~\ref{s:optimization}, we only made some ``first steps'' in optimization over signed tropical numbers, by considering
the tropical quadratic feasibility problem and the unconstrained optimization of tropical polynomials of one variable and tropical quadratic functions. In particular, we showed that the membership in the tropical copositive cone can be checked in polynomial time, whereas the
analogous problem in the classical world is NP-hard. Further relations between tropical
and classical optimization and applications of such relations are to be explored in the future.

We also note that the nature of this work is theoretical and our conclusions and findings are not relevant to any data set.

\section*{Acknowledgement}
The work of S. Sergeev and his visits to CMAP \'{E}cole Polytechnique were partially supported by EPSRC Grant EP/P019676/1.



\end{document}